\documentclass[11pt]{article}
\usepackage[T1]{fontenc}
\usepackage{lmodern}
\usepackage{amsmath,amssymb,amsthm,mathtools}
\usepackage[margin=1in]{geometry}
\usepackage{microtype,xcolor}
\usepackage{tikz}
\usetikzlibrary{arrows.meta,positioning,calc,patterns,decorations.pathreplacing}
\usepackage{needspace}
\usepackage[section]{placeins}
\definecolor{diagBlue}{HTML}{245B8A}
\definecolor{diagOrange}{HTML}{B96422}
\definecolor{diagGreen}{HTML}{337761}
\definecolor{diagGray}{HTML}{606975}
\tikzset{>=Stealth, every picture/.style={line width=.65pt},
  proofbox/.style={draw=diagBlue,fill=diagBlue!5,rounded corners=2pt,
    align=center,inner sep=6pt,font=\small},
  proofarrow/.style={->,draw=diagBlue,line width=.8pt}}

\usepackage[colorlinks=true,linkcolor=blue!45!black,citecolor=blue!45!black,urlcolor=blue!45!black]{hyperref}
\hypersetup{pdftitle={On OpenAI's proof of the Nagata conjecture},pdfauthor={Jenia Tevelev},pdfsubject={Lecture notes, UMass Amherst; version 10},pdfkeywords={Nagata conjecture, plane curves, elliptic ruled surfaces, interpolation}}
\newcommand{\PP}{\mathbb P}
\newcommand{\CC}{\mathbb C}

\newcommand{\ZZ}{\mathbb Z}
\newcommand{\RR}{\mathbb R}
\newcommand{\OO}{\mathcal O}

\newcommand{\cI}{\mathcal I}

\newcommand{\cO}{\mathcal O}

\DeclareMathOperator{\Tot}{Tot}

\DeclareMathOperator{\length}{length}
\DeclareMathOperator{\ord}{ord}
\DeclareMathOperator{\Span}{Span}
\DeclareMathOperator{\rk}{rank}
\newtheorem{maintheorem}{Theorem}

\newtheorem{lemma}{Lemma}[section]
\newtheorem{proposition}[lemma]{Proposition}
\newtheorem{corollary}[lemma]{Corollary}

\theoremstyle{remark}

\theoremstyle{definition}
\newtheorem{definition}{Definition}
\newtheorem{example}[definition]{Example}

\numberwithin{equation}{section}
\allowdisplaybreaks[1]
\title{On OpenAI's proof of the Nagata conjecture}
\author{Jenia Tevelev}
\begin{document}
\maketitle

\begin{abstract}
These notes are an expanded account of my talk at UMass Amherst on OpenAI’s remarkably simple proof of the Nagata conjecture.
I hope they will be useful,
especially to non-experts. The standard reduction to the normal cone and the collision of fat points into one monomial scheme
reduce the problem to interpolation of monomial schemes on the normal
ruled surface of a smooth plane cubic. The novelty is that the study of the asymptotics of theta functions
as the cubic develops a node, together with interpolation duality, reduces
the problem further to interpolation on a triangle of lattice points, which is straightforward. 
The same
calculation determines the nef and pseudoeffective cones of a weighted
blow-up of $\PP^2$ at one point for very general local parameters when
the ratio of weights is greater than nine.
\end{abstract}

\section{Introduction}\label{sec:introduction}

Nagata's conjecture grows out of the
elementary \emph{Lagrange interpolation} problem. We work over~$\CC$.
Let $x_1,\ldots,x_r$ be distinct numbers.
We are looking for a polynomial $f(x)$ with
$f(x_i)=c_i$ for every~$i$, where
$c_1,\ldots,c_r$ are prescribed values.
There is a unique solution of degree at most $r-1$:
$$
f(x)=\sum\limits_{i=1}^r c_i\prod\limits_{t\ne i}\frac{x-x_t}{x_i-x_t}.
$$
The following lemma generalizes Lagrange interpolation.
It will be used in subsequent sections.

\begin{lemma}\label{lem:row-interpolation}
Let $S\subset\CC^2$ be a finite set of points lying on distinct horizontal lines
$y=\beta_0,\ldots,y=\beta_s$, with $N_i$ points on row $i$.
Number the rows so that $N_0\ge\cdots\ge N_s>0$. Set
\begin{equation}\label{eq:intro-row-space}
 I=\{(a,b)\in\ZZ_{\ge0}^2:0\le b\le s,\ 0\le a<N_b\},\quad
 V_I=\Span_{\CC}\{x^ay^b:(a,b)\in I\}.
\end{equation}
The set $I$ is a Young diagram. 
Its rows, counted from the bottom, have
lengths $N_0,\ldots,N_s$. 
For any prescribed values on $S$, there is a unique
polynomial $F(x,y)\in V_I$ taking those values. 
\end{lemma}

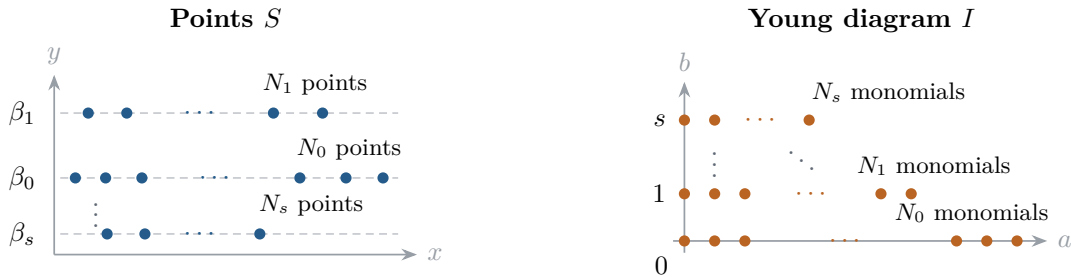
\begin{figure}[!htb]
\centering
\begin{minipage}[t]{.49\linewidth}
\centering
{\small\bfseries Points $S$}\par\smallskip
\begin{tikzpicture}[x=1cm,y=.78cm,font=\small,>=Stealth]
  \draw[->,diagGray!65] (0,-.35)--(4.8,-.35) node[right] {$x$};
  \draw[->,diagGray!65] (0,-.35)--(0,2.7) node[above] {$y$};
  \foreach \yy in {0,.95,2.05}
    \draw[diagGray!35,densely dashed] (.08,\yy)--(4.55,\yy);
  \node[left=3pt] at (0,.95) {$\beta_0$};
  \node[left=3pt] at (0,2.05) {$\beta_1$};
  \node[left=3pt] at (0,0) {$\beta_s$};
  \foreach \xx in {.28,.68,1.16,3.25,3.86,4.35}
    \fill[diagBlue] (\xx,.95) circle (2.1pt);
  \foreach \xx in {.45,.96,2.90,3.55}
    \fill[diagBlue] (\xx,2.05) circle (2.1pt);
  \foreach \xx in {.70,1.2,2.72}
    \fill[diagBlue] (\xx,0) circle (2.1pt);
  \node[diagBlue] at (2.15,.95) {$\cdots$};
  \node[diagBlue] at (1.95,2.05) {$\cdots$};
  \node[diagBlue] at (1.95,0) {$\cdots$};
  \node[font=\footnotesize,above=3pt] at (3.90,.95) {$N_0$ points};
  \node[font=\footnotesize,above=3pt] at (3.45,2.05) {$N_1$ points};
  \node[font=\footnotesize,above=3pt] at (3.40,0) {$N_s$ points};
  \node[diagGray] at (.55,.46) {$\vdots$};
\end{tikzpicture}
\end{minipage}\hfill
\begin{minipage}[t]{.49\linewidth}
\centering
{\small\bfseries Young diagram $I$}\par\smallskip
\begin{tikzpicture}[x=1cm,y=.78cm,font=\small,>=Stealth]
  \draw[->,diagGray!65] (-.15,0)--(4.8,0) node[right] {$a$};
  \draw[->,diagGray!65] (0,-.2)--(0,2.7) node[above] {$b$};
  \foreach \xx in {0,.4,.8,3.6,4.0,4.4}
    \fill[diagOrange] (\xx,0) circle (2.1pt);
  \foreach \xx in {0,.4,.8,2.6,3.0}
    \fill[diagOrange] (\xx,.8) circle (2.1pt);
  \foreach \xx in {0,.4,1.65}
    \fill[diagOrange] (\xx,2.05) circle (2.1pt);
  \node[diagOrange] at (2.15,0) {$\cdots$};
  \node[diagOrange] at (1.7,.8) {$\cdots$};
  \node[diagOrange] at (1.03,2.05) {$\cdots$};
  \node[font=\footnotesize,above=3pt] at (3.8,0) {$N_0$ monomials};
  \node[font=\footnotesize,above=3pt] at (3.3,.8) {$N_1$ monomials};
  \node[font=\footnotesize,above=3pt] at (2.7,2.05) {$N_s$ monomials};
  \node[left=3pt] at (0,.8) {$1$};
  \node[left=3pt] at (0,2.05) {$s$};
  \node[below left=3pt] at (0,0) {$0$};
  \node[diagGray] at (.4,1.43) {$\vdots$};
  \node[diagGray] at (1.55,1.5) {$\ddots$};
\end{tikzpicture}
\end{minipage}
\end{figure}

\begin{proof}
Put $Q_i(y)=\prod\limits_{0\le t<i}(y-\beta_t)$ for $0\le i\le s$, so that $Q_0=1$.
We construct an interpolating polynomial of the form
$F=\sum\limits_{i=0}^s Q_i(y)h_i(x)$
by choosing $h_0,\ldots,h_s$ successively with $\deg h_i<N_i$.
Adding $Q_i(y)h_i(x)$
preserves the values on rows labeled $0,\ldots,i-1$, since $Q_i$
vanishes there. On~row~$i$, the nonzero factor $Q_i(\beta_i)$
reduces the remaining conditions to Lagrange interpolation in $x$
at $N_i$ distinct first coordinates, determining $h_i$ uniquely.
The polynomials $x^aQ_i(y)$, $0\le a<N_i$, form a basis
of $V_I$: their leading terms in $y$ are $x^ay^i$, and the
remaining monomials are indexed by $I$ because $N_t\ge N_i$ for $0\le t<i$.
This proves existence and uniqueness.
\end{proof}

Given points $P_1,\ldots,P_r\in\CC^2$, prescribing values
$F(P_i)=c_i$ for a polynomial $F$ of degree at most $d$ gives
a system of linear equations in its coefficients, generally inhomogeneous.
We now ask whether the associated
homogeneous system $F(P_i)=0$ has a nonzero solution.
Equivalently, we seek a plane algebraic curve $F(x,y)=0$
passing through these points.

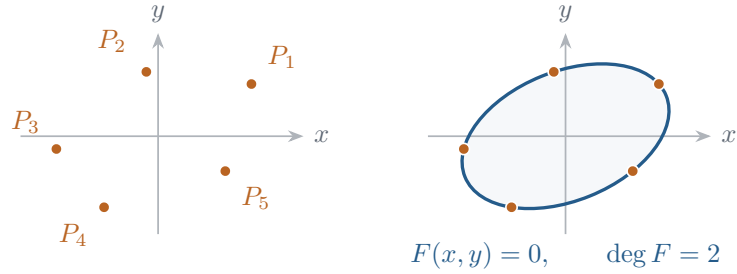
\begin{figure}[htbp]
\centering
\begin{tikzpicture}[x=1cm,y=1cm,scale=.70,font=\small]
  \begin{scope}[xshift=7.7cm,rotate=22]
    \fill[diagBlue!4] (0,0) ellipse [x radius=2.05,y radius=1.22];
  \end{scope}
  \foreach \shift in {0,7.7} {
    \begin{scope}[xshift=\shift cm]
      \draw[->,diagGray!50] (-2.6,0)--(2.75,0) node[right,text=diagGray] {$x$};
      \draw[->,diagGray!50] (0,-1.85)--(0,1.95) node[above,text=diagGray] {$y$};
    \end{scope}
  }

  \begin{scope}[xshift=7.7cm,rotate=22]
    \draw[diagBlue,line width=1.25pt] (0,0) ellipse [x radius=2.05,y radius=1.22];
  \end{scope}

  \foreach \shift in {0,7.7} {
    \begin{scope}[xshift=\shift cm,rotate=22]
      \foreach \angle in {12,83,156,225,297} {
        \fill[white] ({2.05*cos(\angle)},{1.22*sin(\angle)}) circle (3.5pt);
        \fill[diagOrange] ({2.05*cos(\angle)},{1.22*sin(\angle)}) circle (2.6pt);
      }
    \end{scope}
  }
  \begin{scope}[rotate=22]
    \node[above right=3pt,text=diagOrange] at ({2.05*cos(12)},{1.22*sin(12)}) {$P_1$};
    \node[above left=5pt,text=diagOrange] at ({2.05*cos(83)},{1.22*sin(83)}) {$P_2$};
    \node[above left=3pt,text=diagOrange] at ({2.05*cos(156)},{1.22*sin(156)}) {$P_3$};
    \node[below left=3pt,text=diagOrange] at ({2.05*cos(225)},{1.22*sin(225)}) {$P_4$};
    \node[below right=3pt,text=diagOrange] at ({2.05*cos(297)},{1.22*sin(297)}) {$P_5$};
  \end{scope}

   \node[text=diagBlue] at (7.7,-2.25) {$F(x,y)=0,\qquad \deg F=2$};
\end{tikzpicture}
\end{figure}

The polynomials of total degree at most $d$ form a vector space
with basis $x^ay^b$, where $a,b\ge0$ and $a+b\le d$. Its dimension is
$$D=1+2+\cdots+(d+1)=\binom{d+2}{2}.$$
Each condition $F(P_i)=0$ is one homogeneous linear equation
in the coefficients. The corresponding $r\times D$ matrix is
called the \emph{evaluation matrix}; its entry in row $i$ and
the column indexed by $x^ay^b$ is the value of the monomial $x^ay^b$ at $P_i$.
If $r<D$, there are fewer equations than
unknowns, so a nonzero solution exists for \emph{every}
configuration of points. The converse requires a condition on
their position: arbitrarily many points can lie on a curve of fixed positive degree.
To prove that this system has the largest possible rank for
general points, apply Lemma~\ref{lem:row-interpolation} to the triangular set
$$S=T_d=\{(j,k)\in\ZZ_{\ge0}^2:j+k\le d\}.$$
Its row lengths are $d+1,d,\ldots,1$, so $I=T_d$.
Thus evaluation at these $D=\binom{d+2}{2}$ points is an
isomorphism on polynomials of degree at most $d$.
The diagrams below show $d=3$ and $D=10$.

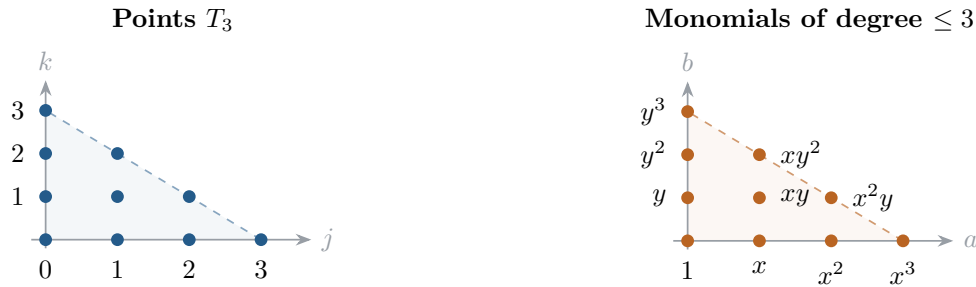
\begin{figure}[!htb]
\centering
\begin{minipage}[t]{.49\linewidth}
\centering
{\small\bfseries Points $T_3$}\par\smallskip
\begin{tikzpicture}[x=.95cm,y=.57cm,font=\small,>=Stealth]
  \fill[diagBlue!5] (0,0)--(3,0)--(0,3)--cycle;
  \draw[->,diagGray!65] (-.2,0)--(3.7,0) node[right] {$j$};
  \draw[->,diagGray!65] (0,-.2)--(0,3.7) node[above] {$k$};
  \draw[diagBlue!55,dashed] (3,0)--(0,3);
  \foreach \xx/\yy in {0/0,1/0,2/0,3/0,0/1,1/1,2/1,0/2,1/2,0/3}
    \fill[diagBlue] (\xx,\yy) circle (2.4pt);
  \foreach \xx in {0,1,2,3}
    \node[below=4pt] at (\xx,0) {$\xx$};
  \foreach \yy in {1,2,3}
    \node[left=4pt] at (0,\yy) {$\yy$};
\end{tikzpicture}
\end{minipage}\hfill
\begin{minipage}[t]{.49\linewidth}
\centering
{\small\bfseries Monomials of degree $\le3$}\par\smallskip
\begin{tikzpicture}[x=.95cm,y=.57cm,font=\small,>=Stealth]
  \fill[diagOrange!5] (0,0)--(3,0)--(0,3)--cycle;
  \draw[->,diagGray!65] (-.2,0)--(3.7,0) node[right] {$a$};
  \draw[->,diagGray!65] (0,-.2)--(0,3.7) node[above] {$b$};
  \draw[diagOrange!65,dashed] (3,0)--(0,3);
  \foreach \xx/\yy in {0/0,1/0,2/0,3/0,0/1,1/1,2/1,0/2,1/2,0/3}
    \fill[diagOrange] (\xx,\yy) circle (2.4pt);
  \node[below=4pt] at (0,0) {$1$};
  \node[below=4pt] at (1,0) {$x$};
  \node[below=4pt] at (2,0) {$x^2$};
  \node[below=4pt] at (3,0) {$x^3$};
  \node[left=4pt] at (0,1) {$y$};
  \node[right=4pt] at (1,1) {$xy$};
  \node[right=4pt] at (2,1) {$x^2y$};
  \node[right=4pt] at (1,2) {$xy^2$};
  \node[left=4pt] at (0,2) {$y^2$};
  \node[left=4pt] at (0,3) {$y^3$};
\end{tikzpicture}
\end{minipage}
\end{figure}

The entries of the evaluation matrix are polynomials in the point
coordinates. For $r\le D$, choose $r$ points of $T_d$; for $r\ge D$,
add $r-D$ distinct points to $T_d$. In either case, the evaluation matrix has
rank $\min(r,D)$, so one of its minors of this size is nonzero.
This minor remains nonzero outside the zero set of a nonzero
polynomial in the point coordinates. A property that holds outside
such a set is said to hold for \emph{general} points.
The kernel therefore has dimension $\max(D-r,0)$, proving
\begin{equation}\label{eq:intro-simple-bound}
 \begin{gathered}
 \text{there is a nonzero polynomial of degree at most $d$}\\
 \text{vanishing at $r$ general points}
 \end{gathered}
 \quad\Longleftrightarrow\quad r<\binom{d+2}{2}.
\end{equation}

We can also prescribe values of the partial derivatives. Prescribing a function's
values and first derivatives, for example, determines the tangent
planes to its graph. These are again linear conditions on the
polynomial's coefficients. As before, we study the homogeneous system
of linear equations, in which the prescribed values and derivatives are zero.

\paragraph{Fat points.}
A polynomial $F$ vanishes to \emph{order at least $m$} at
$P=(x_0,y_0)$ if 
\[
 \left.\frac{\partial^{a+b}F}{\partial x^a\,\partial y^b}
 \right|_{(x,y)=(x_0,y_0)}=0
 \qquad(a,b\ge0,\ a+b<m).
\]
The ideal $(x-x_0,y-y_0)^m\subset\CC[x,y]$ consists of the polynomials
satisfying these conditions and defines the \emph{fat point} of
multiplicity $m$ at $P$. Evaluation at this fat point records the value of the 
polynomial along with all of its partial derivatives of order
less than $m$. Equivalently, it records the class of $F$ in the
quotient ring $\CC[x,y]/(x-x_0,y-y_0)^m$, which is the coordinate algebra of the fat point.
This ring has dimension $\binom{m+1}{2}$ over $\CC$, which is called
the \emph{length} of the fat point. The length is equal to the number of rows a fat point
contributes to the evaluation matrix.

For a second proof of \eqref{eq:intro-simple-bound}, take the fat
point of multiplicity $d+1$ at the origin. Its length is
$D=\binom{d+2}{2}$, and evaluation on polynomials of degree at most
$d$ is an isomorphism: it simply records all coefficients of these polynomials.
By Lemma~\ref{lem:young-smoothing}, this fat point admits a
\emph{flat deformation} into $D$ distinct points. The entries of the
evaluation matrix are regular in the deformation parameter.
Its determinant
is nonzero at the fat point and hence remains nonzero for nearby
parameters. Evaluation at the resulting $D$ distinct points is
therefore an isomorphism. For $r<D$, take a subset of these points;
for $r>D$, add distinct points. The evaluation matrix has rank
$\min(r,D)$ in either case, and the same is true for general points
by the argument with minors above.

\bigskip

Imposing multiplicity $m$ at each of $r$ distinct points gives
$r\binom{m+1}{2}$ homogeneous linear conditions on the coefficients.
Thus, for every configuration of points, a nonzero polynomial of degree
at most $d$ vanishing to the prescribed orders exists whenever
$$ \binom{d+2}{2}>r\binom{m+1}{2}.
$$
If this inequality fails, a nonzero solution may still exist because
the equations may be linearly dependent. For fixed $r$, as $d$ and $m$ grow
proportionally, the leading terms are $d^2/2$ and $rm^2/2$.
This suggests $d/m=\sqrt r$ as the boundary for the existence
of such curves.

\begin{maintheorem}[Nagata's conjecture \cite{nagata1959}]\label{thm:plane}
Let $r\ge10$. For very general points $p_1,\ldots,p_r\in\PP^2$,
every nonzero homogeneous polynomial of degree $d$ vanishing to
order at least $m\ge1$ at each $p_i$ satisfies
\[
 d>m\sqrt r.
\]
\end{maintheorem}

Here \emph{very general} means outside a countable union of proper
algebraic subsets of the configuration space. This allows the
inequality to hold for every pair of integers $d,m$ at the same
configuration of points.
The restriction $r\ge10$ is necessary: nine points always lie on a
nonzero cubic. Nagata proved the conjecture when $r$ is a square
\cite{nagata1959,nagata1965}; see \cite[\S2]{chmr2013} for its
history and related questions.

\paragraph{Scope of these notes.}
I began writing these notes while preparing my talk at UMass Amherst. They are not intended as a
definitive account, and I claim no originality. I have reorganized
and simplified the argument by working directly on the normal ruled
surface of a smooth plane cubic, without performing elementary transformations of ruled surfaces.
Section~\ref{sec:row-interpolation} explains interpolation duality and 
proves the interpolation
statement on a triangle of lattice points.
Section~\ref{sec:elliptic-ruled} realizes this interpolation matrix as
a limit of matrices obtained by evaluating theta functions and their
derivatives. Section~\ref{sec:plane-transfer} transfers the problem  
to the plane and proves Nagata's conjecture.
Section~\ref{sec:one-point} applies the same construction to monomial
valuations and determines the nef and pseudoeffective cones of
weighted blow-ups of $\PP^2$ for very general local parameters,
when the ratio of weights is greater than nine.
I chose to follow the approach using complex geometry in
\cite{openai-nagata2026}, rather than the  algebraic
approach developed in the companion manuscripts
\cite{openai-surfaces2026,openai-higher2026,openai-positive2026},
which I have not studied in detail. The former is also closer
in spirit to our discussions of Nagata's conjecture and related topics,
for example~\cite{tevelev2026}, during the
Summer REU in Algebraic Geometry at UMass Amherst.

\bigskip
As I read the released manuscripts~\cite{openai-release2026}, I wondered whether I might be
their first human reader. They seemed almost entirely unedited.
While it is possible that the release served solely as
a benchmark demonstrating to the investors that a large language model could solve
very hard problems, I suspect that the company expected that 
the mathematical community, either out of curiosity or out of a sense of duty to the profession,
would eventually sort out these contributions.  
Absorbing these arguments into mainstream mathematics and assessing their impact
on our field will require a considerable editorial effort from the
community. This work demands expertise and time, yet brings little
of the recognition associated with original research.
Like many others, I believe this approach harms
our profession, especially the motivation and training of young
people. 

At the same time, I think these manuscripts are a stunning achievement.
The model managed to connect the dots and solve the problems that many of the experts thought were out of reach. 
I think a wholesale re-evaluation of the mathematical corpus is currently underway,
with the old underbrush of mathematical conjectures rapidly being burned in the AI fire. However,
it is up to the mathematical community to figure out what mathematics will look like in the future.

I hope for closer, more respectful collaboration between
frontier AI companies and the academic research community.
The enormous financial resources available to these companies could
support a substantial, independent foundation for research in
mathematics and computer science. Such a foundation could offer
research grants, graduate fellowships, and access to the best frontier models, especially
to young scientists pursuing their own original research. It would give
the community a role in shaping this research and sharing its
benefits, rather than leaving the work to opaque in-house laboratories
with little accountability for its effects on the profession.
I believe this investment would also benefit the AI companies
themselves in the long term.
The foundation could follow the National Science Foundation's model
of peer review and provide sustained support for future generations
of mathematicians and computer scientists.

\paragraph{Acknowledgements.}
I am grateful to the participants in the Reading Seminar in Algebraic
Geometry and 2026 Summer REU  in Algebraic
Geometry at UMass Amherst for many helpful discussions.
This work was supported by the National Science Foundation under
grant DMS-2401387. These notes were edited and illustrated with ChatGPT Astra.

\section{Staircase interpolation}\label{sec:row-interpolation}

The main result of this section is Proposition~\ref{prop:one-point-triangle}, which describes
a vector space of polynomials
in two variables that
interpolates arbitrary values on a certain triangle of lattice points.
We begin by reviewing monomial schemes defined by Young diagrams and their
flat deformations. We then explain
the duality between interpolation on lattice points and vanishing on
monomial schemes.

Let $I\subset\ZZ_{\ge0}^2$ be a Young diagram;
thus $I$ is closed under decreasing either coordinate.
The monomials indexed by points of $\ZZ_{\ge0}^2$ outside $I$ span an ideal in $\CC[w,y]$ defining the finite scheme 
\begin{equation}\label{eq:intro-finite-scheme}
 Z_I=\operatorname{Spec}
 \frac{\CC[w,y]}{(y^a w^b:(a,b)\not\in I)}.
\end{equation}
It is supported at the origin.
Its coordinate ring has a basis consisting of $y^aw^b$, $(a,b)\in I$, so
$Z_I$ has length $|I|$. For the triangle $I=T_m$,
$Z_I$ is the fat point of multiplicity $m+1$, with ideal $(w,y)^{m+1}$.

\begin{lemma}[{\cite[Lemma~18.10]{miller-sturmfels2005}}]\label{lem:young-smoothing}
The scheme $Z_I$ is a flat limit of $|I|$ distinct points.
In particular, a fat point of multiplicity $m$ is a flat limit
of $\binom{m+1}{2}$ distinct points.
\end{lemma}

We will need a closely related result. For positive integers $r,m$, consider the Young diagram
\[
 I=\{(a,b)\in\ZZ_{\ge0}^2:ra+b<rm\}.
\]
We will realize $Z_I$ as a flat limit of $r$ fat points of
multiplicity $m$. Let $X$ be a smooth irreducible complex algebraic
surface, $\mathcal B$ a line bundle on $X$, and
$W\subset H^0(X,\mathcal B)$ a finite-dimensional space of sections.
Choose local holomorphic coordinates $(w,y)$ centered at a point
of $X$ and a local frame of~$\mathcal B$ near that point. In this chart, the
monomials indexed by points outside $I$ generate the ideal
\begin{equation}\label{eq:staircase-scheme}
 \cI_Z=(y^a w^b:ra+b\ge rm)=(w^r,y)^m.
\end{equation}
Thus $Z$ is the scheme $Z_I$ of \eqref{eq:intro-finite-scheme} in
these coordinates. Its local algebra has a basis of $y^aw^b$, $(a,b)\in I$,
and hence
$\length(Z)=|I|=\sum\limits_{a=0}^{m-1}r(m-a)=\frac{rm(m+1)}2$.
Let $H_F(w,y)$ be the local representative of $F\in W$ in the chosen
frame. In the normalized basis $y^aw^b\over a!b!$, restriction to $Z$
is the jet map
\begin{equation}\label{eq:jet-map}
 J:W\longrightarrow\CC^I,\qquad
 (JF)_{a,b}
 =\left.\frac{\partial^{a+b}H_F}{\partial y^a\,\partial w^b}
 \right|_{(w,y)=(0,0)}.
\end{equation}
The collision of fat points in the following lemma is a special case of
\cite[\S3.1, Proposition~5]{evain1998}.

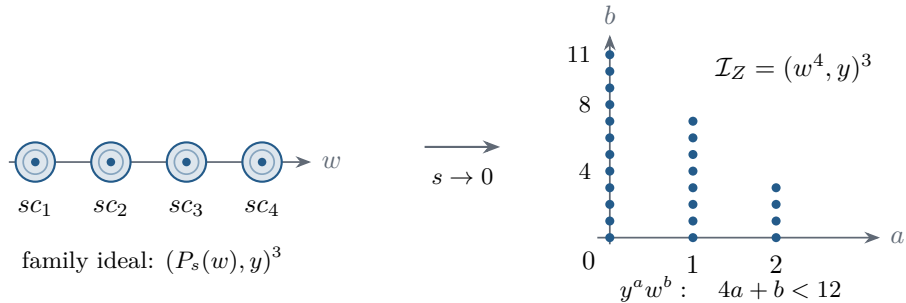
\begin{figure}[htbp]
\centering
\begin{tikzpicture}[x=1cm,y=1cm,font=\small,>=Stealth]
  \draw[->,diagGray] (0.15,1.30)--(4.15,1.30)
    node[right] {$w$};
  \foreach \xx/\ii in {0.50/1,1.50/2,2.50/3,3.50/4}{
    \fill[diagBlue!15] (\xx,1.30) circle (0.26);
    \draw[diagBlue,thick] (\xx,1.30) circle (0.26);
    \draw[diagBlue!55] (\xx,1.30) circle (0.16);
    \fill[diagBlue] (\xx,1.30) circle (1.6pt);
    \node[below] at (\xx,0.93) {$sc_{\ii}$};
  }
  \node[align=center,font=\footnotesize] at (2.05,0.00)
    {family ideal: $(P_s(w),y)^3$};
  \draw[->,thick,diagGray] (5.65,1.50)--(6.65,1.50);
  \node[below,font=\footnotesize] at (6.15,1.35) {$s\to0$};

  \begin{scope}[shift={(8.10,.30)},x=1.10cm,y=.22cm]
    \draw[->,diagGray] (-.15,0)--(3.25,0) node[right] {$a$};
    \draw[->,diagGray] (0,-.4)--(0,12.2) node[above] {$b$};
    \foreach \horizontalindex/\last in {0/11,1/7,2/3}{
      \foreach \verticalindex in {0,...,\last}
        \fill[diagBlue] (\horizontalindex,\verticalindex) circle (1.7pt);
    }
    \foreach \horizontalindex in {1,2}
      \node[below=3pt] at (\horizontalindex,0) {$\horizontalindex$};
    \node[below left=2pt] at (0,0) {$0$};
    \foreach \verticalindex in {4,8,11}
      \node[left=3pt,font=\footnotesize] at (0,\verticalindex) {$\verticalindex$};
    \node[anchor=west,align=left] at (1.15,10.2)
      {$\cI_Z=(w^4,y)^3$};
  \end{scope}
  \node[align=center,font=\footnotesize] at (9.70,-.38)
    {$y^a w^b:\quad 4a+b<12$};
\end{tikzpicture}
\caption{The collision of fat points for $r=4$ and $m=3$.}\label{fig:collision}
\end{figure}

\begin{lemma}\label{lem:collision-open}\label{lem:collision}
The scheme $Z$ is a flat limit of $r$ fat points of multiplicity $m$
with distinct supports. For $r$ general distinct points of $X$,
the restriction of $W$ to the union of the corresponding fat points has rank at least
$\rk J$. In particular, this restriction is injective if $J$ is injective.
\end{lemma}

\begin{proof}
Choose distinct complex numbers $c_1,\ldots,c_r$ and put
$P_s(w)=\prod\limits_{\alpha=1}^r(w-sc_\alpha)$,
where $s$ is a parameter. The polynomial ideal $(P_s(w),y)^m$
defines a finite flat family over the $s$-line, as illustrated in
Figure~\ref{fig:collision}. 
Indeed, there is an
isomorphism of $\CC[s]$-modules
\[
 \frac{\CC[s,w,y]}{(P_s(w),y)^m}
 \simeq
 \bigoplus_{a=0}^{m-1}
 \frac{\CC[s,w]}{(P_s(w)^{m-a})}\,y^a.
\]
Since $P_s(w)^{m-a}$ is monic of degree $r(m-a)$ in $w$,
polynomial division shows that the quotient on the left is free over
$\CC[s]$ with basis $y^aw^b$, $(a,b)\in I$.
Its rank is $rm(m+1)/2$, proving finiteness and flatness
\cite[Tag~00NX]{stacks}.
Analytify this family and restrict the parameter $s$ to a sufficiently small disk
$\mathbb D$ about zero, so that its supports lie in the chosen coordinate
chart. The chart identifies it with a finite flat family
$\mathcal Z\to\mathbb D$ in $X$. At $s=0$, its ideal is $(w^r,y)^m$,
so the special fiber is $Z$. For $s\ne0$, the roots $sc_\alpha$ are
distinct. Near $(sc_\alpha,0)$, the polynomial $P_s(w)$ is $w-sc_\alpha$
times a unit, so the local ideal is $(w-sc_\alpha,y)^m$.
Thus this fiber is the disjoint union of the required $r$ fat points.

Fix a basis of $W$ and use the normalized basis
$y^aw^b\over a!b!$ of the coordinate algebra of the family. Restriction to the family is then represented
by a matrix holomorphic in $s$, whose value at $s=0$ represents $J$.
A nonzero minor of size $\rk J$ remains nonzero for small $s\ne0$.
Thus some configuration of $r$ distinct fat points has restriction
rank at least $\rk J$.
On the space of distinct $r$-point configurations in $X$, the
restriction maps vary algebraically. The same argument with minors
shows that the locus where the rank is at least $\rk J$ is Zariski open.
It is nonempty by the construction, proving the assertion for general points.
\end{proof}

\paragraph{Duality between interpolation and vanishing.}
Let $S\subset\ZZ^2$ be a finite set of lattice points,
and let $I$ be a nonempty finite Young diagram. In
Lemma~\ref{lem:row-interpolation}, $S$ indexes evaluation points and
$I$ indexes the monomials used for interpolation. We now use $S$ to
index Laurent monomials and $I$ to specify vanishing along a monomial
scheme supported at a single point.
On the algebraic torus $(\CC^*)^2$ with coordinates $(\xi,\eta)$, put
$ V_S=\Span_{\CC}\{\eta^j\xi^k:(j,k)\in S\}$.
Set $w=\xi-1$ and $y=\eta-1$, and let $Z_I$ be the monomial scheme
of \eqref{eq:intro-finite-scheme} in these coordinates, supported at
$(1,1)$. Every element of $V_S$ is regular near $(1,1)$, even when
some exponents are negative, so restriction to $Z_I$ is defined.

\begin{lemma}\label{lem:interpolation-duality}
Polynomials in the span of $u^av^b$, $(a,b)\in I$, interpolate
arbitrary values on $S$ if and only if the restriction map
$V_S\to H^0(Z_I,\OO_{Z_I})$ is injective. Equivalently, interpolation on $S$
holds if and only if the zero polynomial is the only Laurent polynomial $G\in V_S$ satisfying
\[
 \left.\frac{\partial^{a+b}G}{\partial\eta^a\,\partial\xi^b}
 \right|_{(\xi,\eta)=(1,1)}=0
 \qquad\text{for every }(a,b)\in I.
\]
\end{lemma}

\begin{proof}
Use the binomial polynomials
$\binom{u}{a}\binom{v}{b}$, $(a,b)\in I$,
in place of the monomials $u^av^b$.
This is a triangular change of basis: each polynomial has
leading term $u^av^b/(a!b!)$, and all its other monomials
are indexed by $I$ because $I$ is a Young diagram.
The evaluation matrix at points of $S$ in this basis has entries
$\binom{j}{a}\binom{k}{b}$ for $(j,k)\in S$ and $(a,b)\in I$, including when $j$ or $k$ is negative.
These are exactly the coefficients of $y^aw^b$ in
$\eta^j\xi^k=(1+y)^j(1+w)^k$.
Thus the restriction matrix in the monomial basis of the
local algebra of $Z_I$ is the transpose of the evaluation
matrix at points of $S$. Consequently, evaluation is surjective if and only
if restriction is injective.
\end{proof}

For ordinary fat points, this duality between jets at a torus point
and interpolation on lattice points appears in
Dumnicki~\cite[Proposition~13]{dumnicki2007},
Gonz\'alez--Karu~\cite[Section~2]{gonzalez-karu2016},
Castravet~\cite[Lemmas~8.2--8.4]{castravet2018}, and
He~\cite{he2021}.
Lemma~\ref{lem:interpolation-duality} gives the same argument for any
Young diagram.

\paragraph{Restriction in holomorphic charts.}
Let $X$ be a smooth irreducible complex algebraic surface, with local
holomorphic coordinates $(\xi,\eta)$ taking the value $(1,1)$ at a
chosen point. Put $w=\log\xi$ and $y=\log\eta$, choosing branches
that vanish at $(1,1)$. Each generator $(\eta-1)^a(\xi-1)^b$ is equal to 
$y^aw^b$ multiplied by a holomorphic unit, so the monomial ideal of $Z_I$
is unchanged in the analytic local ring.
In the normalized basis $y^aw^b/(a!b!)$, restriction of the
germs $\eta^j\xi^k=e^{kw+jy}$ has the matrix
\begin{equation}\label{eq:limit-matrix}
 B_0=\left(
 \left.\frac{\partial^{a+b}e^{kw+jy}}{\partial y^a\,\partial w^b}
 \right|_{(w,y)=(0,0)}
 \right)_{(a,b)\in I,\,(j,k)\in S}
 =\bigl(j^ak^b\bigr)_{(a,b)\in I,\,(j,k)\in S}.
\end{equation}
The matrix $B_0$ is the transpose of the polynomial evaluation matrix on $S$.
The same calculation applies to sections of a holomorphic line bundle on $X$
expressed in a local frame. Suppose bases indexed by $S$ have local representatives
converging to $e^{kw+jy}$, together with the derivatives indexed by the set $I$.
Their restriction matrices then converge to $B_0$. If $B_0$ has full
column rank, one of its $|S|\times|S|$ minors is nonzero and remains
nonzero under sufficiently small perturbations. The nearby restriction
maps are therefore injective. In Section~\ref{sec:elliptic-ruled}, we
apply this observation to bases constructed from theta functions on elliptic curves degenerating to nodal curves.

\paragraph{Weighted monomial ideals.}
Let $p,q,K$ be positive integers. Put
\begin{equation}\label{eq:one-point-ideal}
 \mathfrak J_K^{p,q}(w,y)=(w^b y^a:pb+qa\ge K),\qquad
  I_K^{p,q}=\{(a,b)\in\ZZ_{\ge0}^2:pb+qa<K\}.
\end{equation}
The monomials indexed by $I_K^{p,q}$ form a basis of the quotient algebra
$\CC[w,y]/\mathfrak J_K^{p,q}$. The points in column $t$ of $I_K^{p,q}$ have integer second coordinates satisfying
$0\le b<(K-qt)/p$, so the column has height
\begin{equation}\label{eq:one-point-column-length}
 A_t=\max\left\{0,\left\lceil\frac{K-qt}{p}\right\rceil\right\}.
\end{equation}
For $(p,q)=(1,r)$ and $K=rm$, we recover the ideal of the collision of fat points
in Lemma~\ref{lem:collision}:
\begin{equation}\label{eq:one-point-nagata-ideal}
 \mathfrak J_{rm}^{1,r}(w,y)=(w^r,y)^m.
\end{equation}

\paragraph{The triangle $S_h(\alpha)$.}
Let $h$ be a positive integer and $\alpha$ a positive rational number.
Write $S_h(\alpha)$ for the set of lattice points in the triangle with vertices
$(h,0)$, $(0,\alpha h)$, and $(0,(\alpha+9)h)$, with the upper sloping edge
omitted except for the vertex $(h,0)$:
\begin{equation}\label{eq:one-point-triangle}
 S_h(\alpha)=\{(h,0)\}\ \cup\
 \{(h-e,k):1\le e\le h,\ k\in\ZZ,\ \alpha e\le k<(\alpha+9)e\}.
\end{equation}
For $1\le e\le h$, the column $j=h-e$ contains exactly $9e$ points:
the half-open interval $[\alpha e,(\alpha+9)e)$ has integer length $9e$ and hence
contains $9e$ integers. Including the vertex $(h,0)$, we obtain
\begin{equation}\label{eq:one-point-dimension}
 |S_h(\alpha)|=1+\frac{9h(h+1)}2=\binom{3h+2}{2}.
\end{equation}

\begin{proposition}\label{prop:one-point-triangle}
Let $p,q,h,K$ be positive integers. Set $\alpha=\frac{K}{ph}-9$. Suppose $q>9p$ and
\[
 K\ge3h\sqrt{pq}.
\]
Then $\alpha>0$, and the restriction matrix
\eqref{eq:limit-matrix} for $I=I_K^{p,q}$ and $S=S_h(\alpha)$ has
rank $|S_h(\alpha)|$.
\end{proposition}

\begin{proof}
Put $\sigma=K/(ph)$, so that $\alpha=\sigma-9$.
The hypotheses give
\begin{equation}\label{eq:one-point-balance}
 \sigma\ge3\sqrt{q/p}>9,\qquad \frac{\sigma^2}{9}\ge\frac qp.
\end{equation}
Put $S=S_h(\alpha)$ and $N_k=\#\{j:(j,k)\in S\}$.
List the nonzero row lengths $N_k$ of $S$ in decreasing order to form a
Young diagram. The column $t$ of this diagram has length $\#\{k:N_k\ge t+1\}$.
Thus this diagram is contained in $I_K^{p,q}$ precisely when
\begin{equation}\label{eq:one-point-row-test}
 \#\{k:N_k\ge t\}\le A_{t-1}\qquad(t\ge1),
\end{equation}
where $A_{t-1}$ is the height of the $(t-1)$-st column of $I_K^{p,q}$ given by \eqref{eq:one-point-column-length}.
Under this condition, Lemma~\ref{lem:row-interpolation} gives
interpolation of arbitrary values on $S$ by polynomials in the span
of $u^av^b$, $(a,b)\in I_K^{p,q}$.
By \eqref{eq:limit-matrix}, the restriction matrix $B_0$ with $I=I_K^{p,q}$
therefore has full column rank.

It remains to verify \eqref{eq:one-point-row-test}.
For $0<k<\sigma h$, the horizontal slice of the triangle $S_h(\alpha)$ is
\[
 [\max\{0,h-k/\alpha\},\ h-k/\sigma),
\]
whose width is
\[
 \min\left\{\frac{9k}{\alpha\sigma},\,h-\frac{k}{\sigma}\right\}.
\]
The graph of the width $w(k)$ has a peak at $(\alpha h,9h/\sigma)$; see
Figure~\ref{fig:one-point-width}.

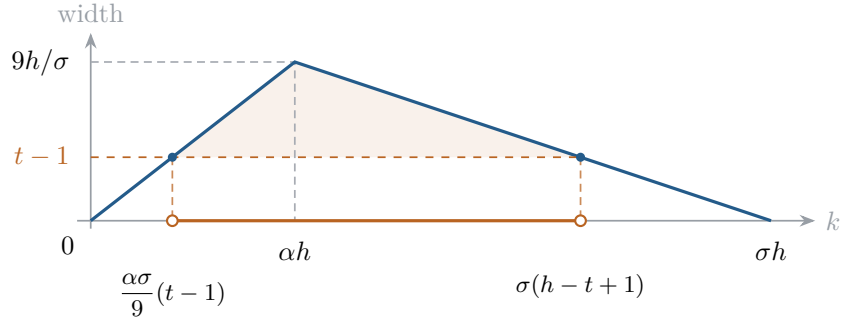
\begin{figure}[!htb]
\centering
\begin{tikzpicture}[x=.90cm,y=.70cm,font=\small,>=Stealth]
  \fill[diagOrange!9] (1.2,1.2)--(3,3)--(7.2,1.2)--cycle;
  \draw[->,diagGray!65] (-.2,0)--(10.65,0) node[right] {$k$};
  \draw[->,diagGray!65] (0,-.15)--(0,3.55) node[above] {width};
  \draw[diagGray!55,densely dashed] (0,3)--(3,3)--(3,0);
  \draw[diagBlue,very thick] (0,0)--(3,3)--(10,0);
  \node[left=4pt] at (0,3) {$9h/\sigma$};
  \node[below=4pt] at (3,0) {$\alpha h$};
  \node[below=4pt] at (10,0) {$\sigma h$};
  \node[below left=3pt] at (0,0) {$0$};
  \draw[diagOrange,dashed] (0,1.2)--(7.2,1.2);
  \node[left=4pt,text=diagOrange] at (0,1.2) {$t-1$};
  \draw[diagOrange!75,densely dashed] (1.2,1.2)--(1.2,0);
  \draw[diagOrange!75,densely dashed] (7.2,1.2)--(7.2,0);
  \fill[diagBlue] (1.2,1.2) circle[radius=1.7pt];
  \fill[diagBlue] (7.2,1.2) circle[radius=1.7pt];
  \draw[diagOrange,very thick] (1.2,0)--(7.2,0);
  \draw[diagOrange,fill=white,thick] (1.2,0) circle[radius=2pt];
  \draw[diagOrange,fill=white,thick] (7.2,0) circle[radius=2pt];
  \node[below=16pt,font=\footnotesize] at (1.2,0)
    {$\dfrac{\alpha\sigma}{9}(t-1)$};
  \node[below=16pt,font=\footnotesize] at (7.2,0)
    {$\sigma(h-t+1)$};
\end{tikzpicture}
\caption{The slice width as a function of height.}
\label{fig:one-point-width}
\end{figure}

For $t\ge2$, a slice containing at least $t$ lattice points has width
strictly greater than $t-1$, since its right endpoint is excluded.
Thus every height with $N_k\ge t$ lies in the open interval
\[
 \frac{\alpha\sigma}{9}(t-1)<k<\sigma(h-t+1).
\]
This interval is illustrated in Figure~\ref{fig:one-point-width}. When nonempty, it has length
\[
 \sigma h-\frac{\sigma^2}{9}(t-1)
 \le\frac{K-q(t-1)}p,
\]
by \eqref{eq:one-point-balance}. An open interval of positive length
$L$ contains at most $\lceil L\rceil$ integers, so there are at most
$A_{t-1}$ such heights by \eqref{eq:one-point-column-length}. If the interval is empty, the bound is automatic.
This proves \eqref{eq:one-point-row-test} for every column of positive index.
For $t=1$, all occupied heights, including the exceptional height zero,
belong to $[0,\sigma h)=[0,K/p)$, which contains
$\lceil K/p\rceil=A_0$ integers. This completes the proof.
\end{proof}

The argument allows equality in the second inequality of
\eqref{eq:one-point-balance}. Thus the proposition includes the case
$K=3h\sqrt{pq}$ whenever this number is an integer. 

\begin{example}[Sixteen points]\label{ex:one-point-square}
Take $(p,q)=(1,16)$, $h=4$, and $K=48=3h\sqrt{pq}$.
Then $\alpha=3$. As shown in
Figure~\ref{fig:one-point-triangle}, $S_4(3)$ has $91$ points,
while $I_{48}^{1,16}$ has $96$ points.
Proposition~\ref{prop:one-point-triangle} says that the resulting
$96\times91$ matrix has rank $91$.
After the transfer to the plane in Section~\ref{sec:plane-transfer},
this rules out a curve of degree $12$
passing through $16$ general points with multiplicity at least three.
This is the equality case
$12=3\sqrt{16}$ in Theorem~\ref{thm:plane}.
Example~\ref{ex:theta-sixteen} gives the corresponding theta series.
\end{example}

\begin{figure}[htbp]
\centering
\begin{minipage}[t]{.49\linewidth}
\centering
{\small\bfseries Lattice points $S_4(3)$}\par\smallskip
\begin{tikzpicture}[x=.88cm,y=.078cm,font=\small,>=Stealth]
  \fill[diagBlue!6] (0,12)--(0,48)--(4,0)--cycle;
  \draw[->,diagGray!65] (-.15,0)--(4.65,0) node[right] {$j$};
  \draw[->,diagGray!65] (0,-2)--(0,53) node[above] {$k$};
  \draw[diagBlue,thick] (0,48)--(0,12)--(4,0);
  \draw[diagBlue,dashed,thick] (0,48)--(4,0);
  \foreach \j/\lo/\hi in {0/12/47,1/9/35,2/6/23,3/3/11}{
    \foreach \k in {\lo,...,\hi}
      \fill[diagBlue] (\j,\k) circle[radius=.9pt];
  }
  \draw[diagBlue,fill=white] (0,48) circle[radius=2pt];
  \fill[diagOrange] (4,0) circle[radius=2.3pt];
  \foreach \j in {1,2,3,4} \node[below=4pt] at (\j,0) {$\j$};
  \foreach \k in {12,24,36,48} \node[left=4pt] at (0,\k) {$\k$};
  \node[below left=3pt] at (0,0) {$0$};
\end{tikzpicture}
\end{minipage}\hfill
\begin{minipage}[t]{.49\linewidth}
\centering
{\small\bfseries Monomials $I_{48}^{1,16}$}\par\smallskip
\begin{tikzpicture}[x=1.15cm,y=.078cm,font=\small,>=Stealth]
  \draw[->,diagGray!65] (-.15,0)--(3.5,0) node[right] {$a$};
  \draw[->,diagGray!65] (0,-2)--(0,53) node[above] {$b$};
  \draw[diagOrange!65,dashed] (0,48)--(3,0);
  \foreach \horizontalindex/\hi in {0/47,1/31,2/15}{
    \foreach \verticalindex in {0,...,\hi}
      \fill[diagOrange] (\horizontalindex,\verticalindex) circle[radius=.9pt];
  }
  \foreach \horizontalindex in {1,2,3} \node[below=4pt] at (\horizontalindex,0) {$\horizontalindex$};
  \foreach \verticalindex in {16,32,48} \node[left=4pt] at (0,\verticalindex) {$\verticalindex$};
  \node[below left=3pt] at (0,0) {$0$};
\end{tikzpicture}
\end{minipage}
\caption{The $91$ lattice points of $S_4(3)$ and the $96$ monomials
indexed by $I_{48}^{1,16}$.}
\label{fig:one-point-triangle}
\end{figure}
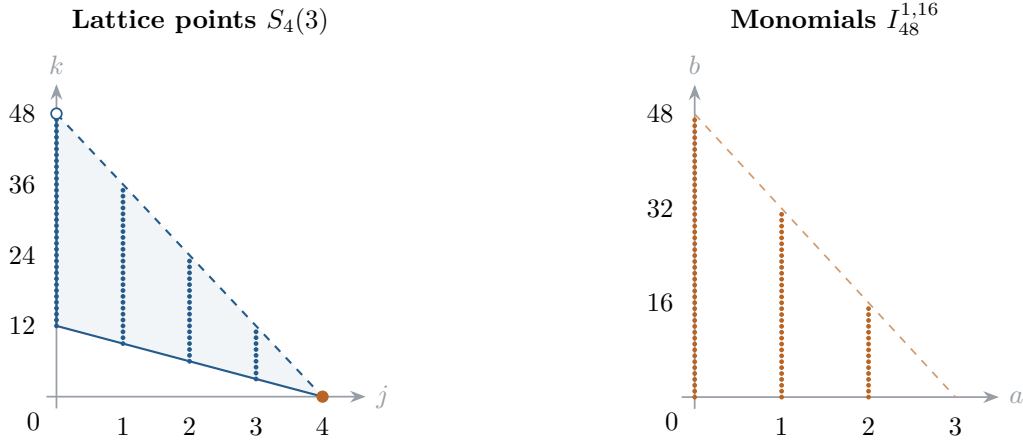

\begin{example}[Ten points]\label{ex:one-point-nonsquare}
Take $(p,q)=(1,10)$, $h=3$, and $K=30$.
Here $3h\sqrt{pq}=9\sqrt{10}<30$, and $\alpha=1$.
Figure~\ref{fig:one-point-ten} shows the $55$ points of $S_3(1)$
and the $60$ monomials indexed by $I_{30}^{1,10}$.
Proposition~\ref{prop:one-point-triangle} gives a $60\times55$
matrix of rank $55$.
Since $\mathfrak J_{30}^{1,10}=(w^{10},y)^3$, this is the interpolation
calculation for ten points of multiplicity three.
Example~\ref{ex:one-point-nagata} explains how it rules out
plane curves of degree $9$ through $10$ general triple points.
Example~\ref{ex:one-point-contact}
extends these conditions to arbitrary multiplicity at one point.
\end{example}

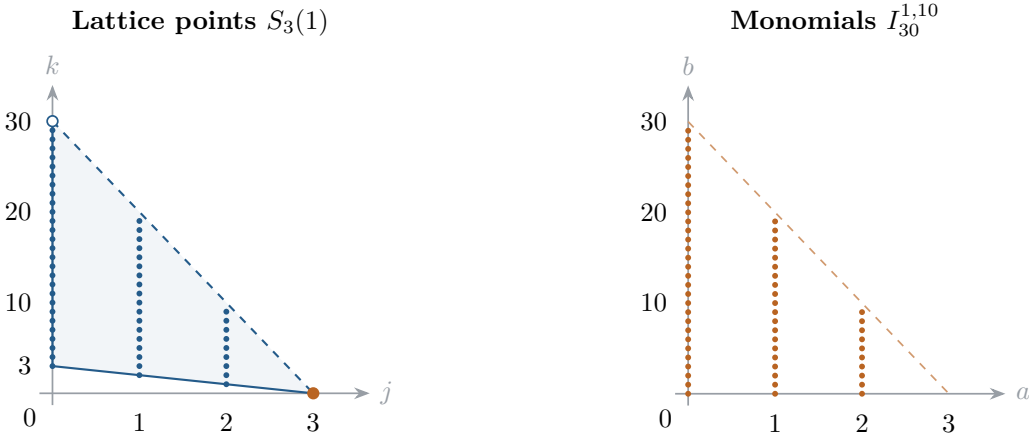
\begin{figure}[!htb]
\centering
\begin{minipage}[t]{.49\linewidth}
\centering
{\small\bfseries Lattice points $S_3(1)$}\par\smallskip
\begin{tikzpicture}[x=1.15cm,y=.12cm,font=\small,>=Stealth]
  \fill[diagBlue!6] (0,3)--(0,30)--(3,0)--cycle;
  \draw[->,diagGray!65] (-.15,0)--(3.65,0) node[right] {$j$};
  \draw[->,diagGray!65] (0,-1.3)--(0,34) node[above] {$k$};
  \draw[diagBlue,thick] (0,30)--(0,3)--(3,0);
  \draw[diagBlue,dashed,thick] (0,30)--(3,0);
  \foreach \j/\lo/\hi in {0/3/29,1/2/19,2/1/9}{
    \foreach \k in {\lo,...,\hi}
      \fill[diagBlue] (\j,\k) circle[radius=1.1pt];
  }
  \draw[diagBlue,fill=white] (0,30) circle[radius=2pt];
  \fill[diagOrange] (3,0) circle[radius=2.3pt];
  \foreach \j in {1,2,3} \node[below=4pt] at (\j,0) {$\j$};
  \foreach \k in {3,10,20,30} \node[left=4pt] at (0,\k) {$\k$};
  \node[below left=3pt] at (0,0) {$0$};
\end{tikzpicture}
\end{minipage}\hfill
\begin{minipage}[t]{.49\linewidth}
\centering
{\small\bfseries Monomials $I_{30}^{1,10}$}\par\smallskip
\begin{tikzpicture}[x=1.15cm,y=.12cm,font=\small,>=Stealth]
  \draw[->,diagGray!65] (-.15,0)--(3.65,0) node[right] {$a$};
  \draw[->,diagGray!65] (0,-1.3)--(0,34) node[above] {$b$};
  \draw[diagOrange!65,dashed] (0,30)--(3,0);
  \foreach \horizontalindex/\hi in {0/29,1/19,2/9}{
    \foreach \verticalindex in {0,...,\hi}
      \fill[diagOrange] (\horizontalindex,\verticalindex) circle[radius=1.1pt];
  }
  \foreach \horizontalindex in {1,2,3} \node[below=4pt] at (\horizontalindex,0) {$\horizontalindex$};
  \foreach \verticalindex in {10,20,30} \node[left=4pt] at (0,\verticalindex) {$\verticalindex$};
  \node[below left=3pt] at (0,0) {$0$};
\end{tikzpicture}
\end{minipage}
\caption{The $55$ lattice points of $S_3(1)$ and the $60$ monomials
indexed by $I_{30}^{1,10}$ in Example~\ref{ex:one-point-nonsquare}.}
\label{fig:one-point-ten}
\end{figure}

\section{Interpolation on the normal ruled surface}
\label{sec:elliptic-ruled}

Let $C\subset\PP^2$ be a smooth plane cubic, with normal bundle
$M=N_{C/\PP^2}=\OO_C(3)$ of degree nine. 
Let $\Tot(M)$ be its total space.
Its projective completion is a ruled surface
\[
 \pi:\bar\Sigma=\PP_C(\OO_C\oplus M)\longrightarrow C.
\]
We use the convention that projective bundles parametrize lines rather than quotients.
The summands $\OO_C$ and $M$ of $\OO_C\oplus M$ define the zero section $C_0$ and the
infinity section $C_\infty$ of $\bar\Sigma$, respectively, so
$\bar\Sigma\setminus C_\infty=\Tot(M)$. The projective bundle formulas give
\begin{equation}\label{eq:normal-sections}
 \OO(C_0)=\OO(C_\infty)\otimes\pi^*M,\qquad
 C_0^2=9,\quad C_\infty^2=-9,\quad C_0\cdot C_\infty=0.
\end{equation}
See \cite[II, Proposition~7.11; V, Propositions~2.3 and~2.9]{hartshorne1977}.

\begin{maintheorem}\label{thm:ruled}
Let $p,q,h,K$ be positive integers with $q>9p$ and
$K\ge3h\sqrt{pq}$. There exist a smooth plane cubic $C$, a point
$x\in\Tot(M)$, and local holomorphic coordinates $(w,y)$ centered
at $x$ such that
\begin{equation}\label{eq:one-point-ruled-vanishing}
 H^0\!\left(\bar\Sigma,\OO(hC_0)\otimes
                      \mathfrak J_K^{p,q}(w,y)\right)=0.
\end{equation}
Here $\mathfrak J_K^{p,q}(w,y)$ denotes the ideal sheaf that
extends the local ideal by the unit ideal away from $x$.
\end{maintheorem}

In the next sections, we will show that the 
specialization $p=1$, $q=r$, $K=rm$ gives
Theorem~\ref{thm:plane} (Nagata Conjecture) and that the general case gives
Theorem~\ref{thm:one-point} of Section~\ref{sec:one-point}
about the weighted blow-up of $\PP^2$ with weights $(p,q)$ at one point for very general local
parameters.

\paragraph{Strategy of the proof.}
Let $v$ denote the tautological section
of $\pi^*M$ on $\Tot(M)$, which becomes a vertical fiber coordinate in a local frame of $M$ on $C$.
On $\Tot(M)$, we have $\OO(hC_0)\cong\pi^*M^h$ by \eqref{eq:normal-sections}.
We have 
$$H^0(\Tot(M), \pi^*M^h)=H^0(C,\pi_*\pi^*M^h)\cong H^0(C,M^h\otimes\pi_*\cO)\cong\bigoplus_{j=0}^\infty H^0(C,M^{h-j}).$$
More precisely, every global section of $\pi^*M^h$ on $\Tot(M)$ can be written uniquely as 
$\sum\limits_{j=0}^\infty (\pi^*g_j)v^j$, where
$g_j\in H^0(C,M^{h-j})$.
On $\bar\Sigma$, 
we have $\OO(hC_0)\cong(\pi^*M^h)(hC_\infty)$ by \eqref{eq:normal-sections}.
Thus every global section of $\OO(hC_0)$ on $\bar\Sigma$
restricts to $\Tot(M)$ as a unique sum
\begin{equation}\label{eq:one-point-coefficients}
 F=\sum_{j=0}^h (\pi^*g_j)v^j,\qquad
 g_j\in H^0(C,M^{h-j}).
\end{equation}
The summation only goes up to $h$ because the order of the poles along $C_\infty$
is bounded from above by $h$.
The Riemann--Roch theorem gives $h^0(C,M^{h-j})=9(h-j)$ for $j<h$, while
$h^0(C,\OO_C)=1$ \cite[IV, Theorem~1.3]{hartshorne1977}. Hence
\[
 h^0(\bar\Sigma,\OO(hC_0))=1+\frac{9h(h+1)}2=\binom{3h+2}{2}.
\]
This is the cardinality of the set $S_h(\alpha)$ by \eqref{eq:one-point-dimension}.
For a carefully chosen family of cubic curves degenerating to a nodal cubic,
we will construct 
a point $x\in\Tot(M)$, local holomorphic coordinates $(w,y)$ centered at~$x$, and 
a basis of $H^0(\bar\Sigma,\OO(hC_0))$.
In the notation of \eqref{eq:one-point-coefficients}, the basis elements will be
 of the form $(\pi^*g_k)v^j$ with 
$g_k\in H^0(C,M^{h-j})$
indexed by points $(j,k)\in S_h(\alpha)$.
The local
expressions of the basis elements at the point $x$ will converge to $e^{kw+jy}$ as the cubic degenerates, 
with every partial derivative converging to the corresponding derivative of this exponential.
The restriction matrices will  converge to the matrix $B_0$ in
\eqref{eq:limit-matrix}, and Proposition~\ref{prop:one-point-triangle}
will give \eqref{eq:one-point-ruled-vanishing}.

\bigskip

The GAGA theorems allow us to construct line bundles and compute global sections
in the holomorphic category \cite[\S12, Theorems~1--3]{serre1956}.
For $0<\tau<1$, consider the Tate curve
\cite[V, \S1]{silverman1994}
\[
 C_\tau=\CC^*/\tau^\ZZ
       \simeq\CC/(2\pi i\ZZ+(\log\tau)\ZZ).
\]
The additive and multiplicative quotients
are isomorphic by the exponential map.
We denote the quotient map $\CC^*\to C_\tau$ by $z\mapsto [z]$.
The annulus $\tau\le |z|\le1$ is a fundamental domain in $\CC^*$, with its outer
boundary identified with its inner boundary by the dilation $z\mapsto\tau z$.

For $d\in\ZZ$ and $\gamma\in\CC^*$, let $\mathcal B_{d,\gamma}$
be the line bundle on $C_\tau$ obtained from the trivial bundle
$\CC^*\times\CC$ on $\CC^*$ by the $\ZZ$-action generated by
$
 (z,t)\longmapsto(\tau z,\gamma z^{-d}t)$.
In the covering frame, a holomorphic section of
$\mathcal B_{d,\gamma}$ lifts uniquely to a holomorphic function
$f\colon\CC^*\to\CC$ satisfying
\[
 f(\tau z)=\gamma z^{-d}f(z).
\]
Functions of this form are called {\em multiplicative theta functions of order $d$}.
Conversely, every such function descends to a section.
Multiplication of the covering frames gives
\[
 \mathcal B_{d,\gamma}\otimes\mathcal B_{d',\gamma'}
 \simeq\mathcal B_{d+d',\gamma\gamma'}.
\]
The multiplicative theta function $ \Theta_\tau(z)$ is defined as the infinite product
\begin{equation}\label{ra:theta-product}
 \Theta_\tau(z)=\prod_{s=0}^{\infty}(1-\tau^sz)
                     \prod_{s=1}^{\infty}(1-\tau^s/z),
\end{equation}
which converges normally on compact annuli,
has simple zeros exactly at $\tau^\ZZ$, and satisfies
\[
 \Theta_\tau(\tau z)=-z^{-1}\Theta_\tau(z),
\]
as described in \cite[\S4.1.2]{rosengren2020}.
Thus $\Theta_\tau(-z/\gamma)$ lifts a section of
$\mathcal B_{1,\gamma}$ with divisor $[-\gamma]\in C_\tau$.
In particular, $\deg\mathcal B_{1,\gamma}=1$, and the tensor
product rule gives $\deg\mathcal B_{d,\gamma}=d$ for every integer $d$.

\begin{lemma}\label{lem:theta-basis}
If $d>0$, then $H^0(C_\tau,\mathcal B_{d,\gamma})$ has
dimension $d$. Write the equivariant lift of a section as a
Laurent series
$
 f(z)=\sum\limits_{k\in\ZZ}c_kz^k
$.
Its coefficients satisfy the recurrence
\begin{equation}\label{qwrfqerqergtq3t}
c_{k+d}=\gamma^{-1}\tau^kc_k
\end{equation}
for every $k\in\ZZ$. A basis is given by the sections whose
equivariant lifts are the Laurent theta series
\begin{equation}\label{ra:theta-basis-functions}
 \vartheta_k^{d,\gamma}(z)
   =\sum_{s\in\ZZ}\gamma^{-s}
               \tau^{sk+{ds(s-1)\over 2}}z^{k+ds},
\end{equation}
where we choose one $k$ in each residue class modulo $d$.
\end{lemma}

\begin{proof}
Comparing coefficients in the Laurent series for
$f(\tau z)=\gamma z^{-d}f(z)$ gives the recurrence \eqref{qwrfqerqergtq3t}.
Thus one coefficient in each residue class modulo $d$ determines all the others.
These $d$ coefficients can be chosen arbitrarily. Indeed,
on a compact annulus, the terms of each series
\eqref{ra:theta-basis-functions} are bounded by
$c_1\exp(-c_2s^2+c_3|s|)$ for constants $c_1,c_2>0$ and $c_3$, since
$d>0$ and $\log\tau<0$. This bound is summable over $s\in\ZZ$:
for large $|s|$, it is at most $c_1\exp(-c_2s^2/2)$.
The sum is holomorphic
\cite[Chapter~2, Theorem~5.2]{stein-shakarchi2003},
and reindexing the series gives
$f(\tau z)=\gamma z^{-d}f(z)$.
Thus $f(z)$ descends to a global section of $\mathcal B_{d,\gamma}$. 
This proves the basis claim.
\end{proof}

As $\tau\to0$, the Tate curve $C_\tau$ degenerates to a nodal curve
whose normalization is $\PP^1$ with $0$ and $\infty$ identified
\cite[\S5.2, Proposition~5.7 and Exercise~47]{hain2011}.
To motivate the subsequent choices, we recall that the nodal degeneration and its vanishing cycle
 can be seen in the {\em plumbing chart}
$$X=z,\quad Y=\tau/z,
 \qquad |X|,|Y|<\varepsilon.
$$
Here we fix $\varepsilon$ such that
 \(0<\varepsilon<1\) and \(0<\tau<\varepsilon^2\).
 These inequalities guarantee that the plumbing chart is contained in the fundamental annulus.
 The latter is diffeomorphic to the cylinder
\[
 z=e^{r+i\theta},\qquad \log\tau\le r\le0,
 \qquad \theta\in\RR/2\pi\ZZ.
\]
The angular circle
$\gamma_\tau(\theta)=[\sqrt\tau e^{i\theta}]$ in $C_\tau$ has coordinates
$(X,Y)=(\sqrt\tau e^{i\theta},\sqrt\tau e^{-i\theta})$ and shrinks into the node $(0,0)$ as $\tau\to0$.
It is the vanishing cycle in Figure~\ref{fig:elliptic-plumbing}.

\begin{figure}[!htbp]
\centering
\begin{tikzpicture}[x=1cm,y=1cm,font=\footnotesize,>=Stealth,
  line cap=round,line join=round]
  \begin{scope}
    \node[font=\small\bfseries] at (2.2,2.70)
      {(a) The  cylinder};
    \node at (2.2,2.20) {$z=e^{r+i\theta}$};
    \fill[diagBlue!6] (.45,-.32) rectangle (3.95,1.12);
    \draw[diagBlue,thick] (.45,1.12)--(3.95,1.12);
    \draw[diagBlue,thick] (.45,-.32)--(3.95,-.32);
    \draw[diagBlue,thick] (.45,.40) ellipse[x radius=.22,y radius=.72];
    \draw[diagBlue,thick] (3.95,.40) ellipse[x radius=.22,y radius=.72];
    \draw[diagGreen,thick,dashed]
      (2.2,1.12) arc[start angle=90,end angle=270,x radius=.22,y radius=.72];
    \draw[diagGreen,very thick]
      (2.2,-.32) arc[start angle=-90,end angle=90,x radius=.22,y radius=.72];
    \node[text=diagGreen] at (2.2,1.33) {$\gamma_\tau$};
    \draw[->,diagOrange,very thick] (.65,.08)--(4.15,.08);
    \fill[diagOrange] (.65,.08) circle[radius=1.7pt];
    \fill[diagOrange] (4.15,.08) circle[radius=1.7pt];
    \node[below] at (.45,-.48) {$\log\tau$};
    \node[below] at (2.2,-.48) {$\frac12\log\tau$};
    \node[below] at (3.95,-.48) {$0$};
  \end{scope}
  \begin{scope}[xshift=5.05cm]
    \node[font=\small\bfseries] at (2.2,2.70)
      {(b) The plumbing chart};
    \node at (2.2,2.20) {$(X,Y)=(z,\tau/z)$};
    \fill[diagBlue!6]
      (.40,1.30)
      .. controls (1.12,1.30) and (1.45,.75) .. (2.20,.75)
      .. controls (2.95,.75) and (3.28,1.30) .. (4.00,1.30)
      --(4.00,-.50)
      .. controls (3.28,-.50) and (2.95,.05) .. (2.20,.05)
      .. controls (1.45,.05) and (1.12,-.50) .. (.40,-.50)--cycle;
    \draw[diagBlue,thick]
      (.40,1.30)
      .. controls (1.12,1.30) and (1.45,.75) .. (2.20,.75)
      .. controls (2.95,.75) and (3.28,1.30) .. (4.00,1.30);
    \draw[diagBlue,thick]
      (.40,-.50)
      .. controls (1.12,-.50) and (1.45,.05) .. (2.20,.05)
      .. controls (2.95,.05) and (3.28,-.50) .. (4.00,-.50);
    \draw[diagBlue,thick] (.40,.40) ellipse[x radius=.19,y radius=.90];
    \draw[diagBlue,thick] (4.00,.40) ellipse[x radius=.19,y radius=.90];
    \node at (.50,1.63) {$|Y|=\varepsilon$};
    \node at (3.90,1.63) {$|X|=\varepsilon$};
    \draw[diagGreen,thick,dashed]
      (2.20,.75) arc[start angle=90,end angle=270,x radius=.14,y radius=.35];
    \draw[diagGreen,very thick]
      (2.20,.05) arc[start angle=-90,end angle=90,x radius=.14,y radius=.35];
    \node[text=diagGreen] at (2.20,-.28) {$\gamma_\tau$};
    \node at (2.20,-.89) {$|X|=|Y|=\sqrt\tau$};
    \end{scope}
  \begin{scope}[xshift=10.30cm]
    \node[font=\small\bfseries] at (2.15,2.70)
      {(c) The nodal fiber};
    \node at (2.15,2.20) {$XY=0$};
    \draw[diagBlue,very thick] (.70,-.35)--(3.60,1.15);
    \draw[diagBlue,very thick] (.70,1.15)--(3.60,-.35);
    \node[above right] at (3.48,1.10) {$Y=0$};
    \node[below right] at (3.48,-.30) {$X=0$};
    \fill[diagGreen] (2.15,.40) circle[radius=3pt];
    \draw[diagGreen] (2.15,.30)--(2.15,-.65);
    \node[text=diagGreen] at (2.15,-.89) {the collapsed cycle};
  \end{scope}
\end{tikzpicture}
\caption{The vanishing cycle $\gamma_\tau$.}
\label{fig:elliptic-plumbing}
\end{figure}
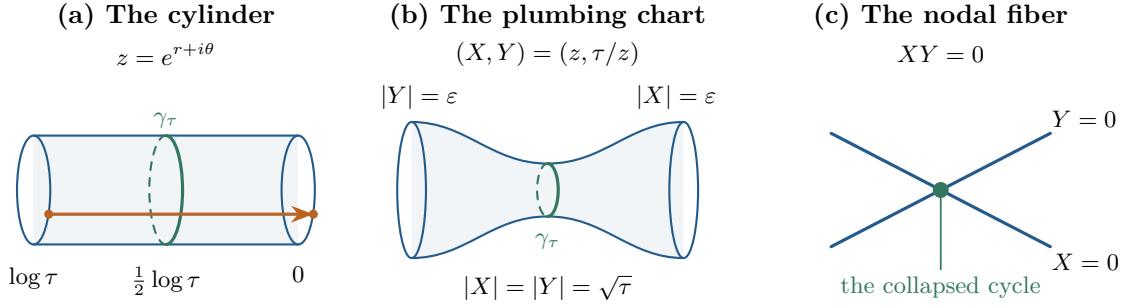
A fixed small $w$-disk centered at the origin gives the \emph{moving chart} 
$z=\sqrt\tau e^w$
of $C_\tau$ centered at
$z_0=\sqrt\tau$ for all sufficiently small $\tau$. This point approaches the node as $\tau\to0$.

We now normalize the theta basis of Lemma~\ref{lem:theta-basis} in the moving chart. 

\begin{lemma}\label{lem:theta-limit}
Let $d$ be a positive integer and let $U\in\RR\setminus\ZZ$.
The bundle $\mathcal B_{d,\tau^{U+d/2}}$ has a basis,
indexed by the $d$ integers $U<k<U+d$, whose equivariant lifts satisfy
\begin{equation}\label{ra:normalized-theta}
 \widehat{\vartheta}_{k,\tau}(\sqrt\tau e^w)
 =\sum_{s\in\ZZ}
      \tau^{s(k-U)+{ds(s-1)\over 2}}e^{(k+ds)w}
 \xrightarrow[\quad\tau\to 0\quad]{} e^{kw}.
\end{equation}
The convergence is uniform on compact subsets of $\CC$, with
every fixed derivative.
\end{lemma}

\begin{proof}
The interval $(U,U+d)$ contains one integer in each residue class
modulo $d$. In Lemma~\ref{lem:theta-basis}, take
$\gamma=\tau^{U+d/2}$ and normalize the theta basis by setting
\[
 \widehat{\vartheta}_{k,\tau}(z)
 =\tau^{-k/2}\vartheta_k^{d,\gamma}(z).
\]
Substituting $z=\sqrt\tau e^w$ gives the normalized theta series
\eqref{ra:normalized-theta}.

The term with $s=0$ is $e^{kw}$. For $s\ge1$, the powers of $\tau$
in the terms $s$ and $-s$ are
\[
 s(k-U)+\frac{ds(s-1)}2,\qquad
 s(U+d-k)+\frac{ds(s-1)}2,
\]
respectively.
Both are positive. On the disc $|w|\le R$, the sum of the remaining
terms is therefore bounded by
\[
 e^{|k|R}\sum_{s\ge1}
 \left[(e^{dR}\tau^{k-U})^s+
       (e^{dR}\tau^{U+d-k})^s\right]\xrightarrow[\tau\to0]{}0.
\]
These are geometric series once $\tau$ is small.
To control derivatives, write
$g_\tau(w)=\widehat{\vartheta}_{k,\tau}(\sqrt\tau e^w)-e^{kw}$.
For $R,\rho>0$, apply the same bound on $|w|\le R+\rho$.
Cauchy's estimate on a circle of radius $\rho$ about each point
with $|w|\le R$ gives, for every $\ell\ge0$,
\begin{equation}\label{eq:theta-derivative-estimate}
 \sup_{|w|\le R}\left|\frac{\partial^\ell g_\tau}{\partial w^\ell}(w)\right|
 \le \frac{\ell!}{\rho^\ell}
       \sup_{|\zeta|\le R+\rho}|g_\tau(\zeta)|
 \xrightarrow[\tau\to0]{}0.
\end{equation}
See \cite[Chapter~2, Corollary~4.3 and Theorem~5.3]{stein-shakarchi2003}.
\end{proof}

\begin{proof}[Proof of Theorem~\ref{thm:ruled}]
Put $\alpha=K/(ph)-9>0$, as in Proposition~\ref{prop:one-point-triangle}.
We are looking for a line bundle $M_\tau$ on $C_\tau$ of degree $9$ so that
normalized global sections of $M_\tau^e$, a line bundle of degree \(9e\),
have asymptotics $e^{kw}$ for $k$ in the \(j=h-e\) column
of the triangle $S_h(\alpha)$, which requires
$
\alpha e\le k<(\alpha+9)e
$
by \eqref{eq:one-point-triangle}.
By Lemma~\ref{lem:theta-limit}, the line bundle
$
\mathcal B_{d,\tau^{U+d/2}}
$
has normalized theta functions tending to \(e^{kw}\) for the integers
$
U<k<U+d
$.
Thus we need to take
$$
M_\tau=\mathcal B_{9,\tau^{\alpha+9/2-\delta}},
$$
where the small correction \(-\delta\) is there to handle the half-open endpoints. Concretely,
write $\alpha=A/D$ in lowest terms and put
$ \delta=\frac1{2Dh}$.
For $1\le e\le h$, 
we have
$
 M_\tau^e=\mathcal B_{9e,\,\tau^{9e/2+e(\alpha-\delta)}}$.
Lemma~\ref{lem:theta-limit} applies with $U=e(\alpha-\delta)$ and degree
$9e$. Since $\alpha e\in D^{-1}\ZZ$ and $0<e\delta<1/D$, we have
\begin{equation}\label{eq:one-point-theta-interval}
 (e(\alpha-\delta),e(\alpha-\delta)+9e)\cap\ZZ
 =[\alpha e,(\alpha+9)e)\cap\ZZ.
\end{equation}
In particular, the endpoints on the left are nonintegral. By \eqref{eq:one-point-triangle},
the normalized
theta functions of $M_\tau^e$ therefore tend to $e^{kw}$ for exactly
the integers $k$ occurring in column $j=h-e$ of $S_h(\alpha)$.

We work in the fixed covering frame of $M_\tau$ and its tensor powers.
The chart $(z,v)=(\sqrt\tau e^w, e^y)$ embeds a fixed polydisk $|w|,|y|<R<\pi$ into $\Tot M_\tau$
once $|\log\tau|>2R$. Indeed, if two points $(w,y)$ and $(w',y')$ have the same image, then $w'-w=s\log\tau+2\pi it$ for integers $s,t$.
The bounds on real and imaginary parts force $s=t=0$.
The map $y\mapsto e^y$ is injective on this disk as well.

Let $\bar\Sigma_\tau=\PP_{C_\tau}(\OO\oplus M_\tau)$.
For each $1\le e\le h$, multiply the normalized theta basis of
$H^0(C_\tau,M_\tau^e)$ by $v^{h-e}$.
Together with $v^h$, these sections form a basis of
$H^0(\bar\Sigma_\tau,\OO(hC_0))$
by \eqref{eq:one-point-coefficients}.
The basis is indexed by $S_h(\alpha)$: the sections coming from
$M_\tau^e$ correspond to $(j,k)$ with $j=h-e$ and
$\alpha e\le k<(\alpha+9)e$, and $v^h$ corresponds to $(h,0)$.

In the chart $(z,v)=(\sqrt\tau e^w,e^y)$ of $\Tot(M_\tau)$, using the covering
frame of $M_\tau$ and its tensor powers,
the section indexed by $(j,k)$ converges locally to $e^{kw+jy}$.
By Lemma~\ref{lem:theta-limit} and
\eqref{eq:theta-derivative-estimate}, this convergence holds
uniformly on compact subsets, together with every derivative
in $w$.
The dependence on $y$ is entirely through the factor $e^{jy}$,
so differentiating $a$ times in $y$ simply multiplies the
local expression by $j^a$.
Thus the mixed derivative
$\partial^{a+b}/(\partial y^a\partial w^b)$,
evaluated at $(w,y)=(0,0)$, converges to $j^a k^b$.

The evaluation matrix therefore converges entrywise to
\eqref{eq:limit-matrix}, with basis sections indexed by
$S_h(\alpha)$ and derivatives indexed by $I_K^{p,q}$.
By Proposition~\ref{prop:one-point-triangle}, the limiting
matrix has a nonzero minor of size $|S_h(\alpha)|$.
This minor remains nonzero for all sufficiently small
positive $\tau$, proving
\eqref{eq:one-point-ruled-vanishing}.

The line bundles $M_\tau$ are normal bundles of plane cubics. Indeed,
the degree-three line bundle
$
 L_\tau=\mathcal B_{3,\,\tau^{3/2+(\alpha-\delta)/3}}$
 satisfies $L_\tau^3=M_\tau$
and embeds $C_\tau$ as a smooth plane cubic,
with normal bundle $L_\tau^3=M_\tau$
\cite[IV, Example~3.3.3]{hartshorne1977}; see also
\cite[Tags~0H2V and~0B3P]{stacks}. 
\end{proof}

\begin{example}[Theta series for sixteen points]\label{ex:theta-sixteen}
We return to Example~\ref{ex:one-point-square} and
Figure~\ref{fig:one-point-triangle}.
Here $(p,q)=(1,16)$, $h=4$, and $K=48$, so $\alpha=3$,
$\delta=1/8$, and $
 M_\tau=\mathcal B_{9,\tau^{59/8}}$.
We describe explicitly the basis theta function associated with
each point of $S_4(3)$.
For $1\le e\le4$, the points in column $j=4-e$ are
$(4-e,k)$ for $3e\le k<12e$.
The corresponding normalized theta functions, multiplied by
the fiber factor $v^{4-e}$, have local expressions
\[
 \sum_{t\in\ZZ}
 \tau^{t(k-23e/8)+{9e\,t(t-1)\over2}}
 e^{(k+9et)w+(4-e)y}.
\]
The term with $t=0$ is $e^{kw+(4-e)y}$, whose exponents
record the lattice point $(4-e,k)$.
Every other term has a positive power of $\tau$.
By Lemma~\ref{lem:theta-limit}, the sum of these remaining
terms tends to zero uniformly on compact subsets, together
with all derivatives.
For example, the point $(3,3)$ corresponds to $e=1$ and $k=3$.
Displaying the terms with $t=0,1,-1$, we obtain
\[
 e^{3w+3y}
 +\tau^{1/8}e^{12w+3y}
 +\tau^{71/8}e^{-6w+3y}
 +\cdots
 \longrightarrow e^{3w+3y}.
\]
\end{example}

\section{Proof of the Nagata Conjecture}
\label{sec:plane-transfer}

We first transfer Theorem~\ref{thm:ruled} from the ruled surface to the plane. 
We fix
a smooth cubic for which Theorem~\ref{thm:ruled} holds and then degenerate
the plane to its normal cone. 

\begin{proposition}\label{prop:one-point-transfer}
Fix positive integers $p,q,h,K$. Let $C\subset\PP^2$ be a smooth cubic
with normal bundle~$M$, and let
$\bar\Sigma=\PP_C(\OO_C\oplus M)$. Suppose
\[
 H^0\!\left(\bar\Sigma,\OO(hC_0)\otimes
                      \mathfrak J_K^{p,q}(w,y)\right)=0
\]
for local holomorphic coordinates at a point of $\Tot(M)$.
Then there are regular local parameters $(u,v)$ at a point of
$\PP^2$ for which
\begin{equation}\label{eq:one-point-plane-vanishing}
 H^0\!\left(\PP^2,\OO(3h)\otimes\mathfrak J_K^{p,q}(u,v)\right)=0.
\end{equation}
\end{proposition}

\begin{proof}
We use a standard deformation to the normal cone
\cite[\S5.1]{fulton1998}; see also
\cite[II, Theorem~8.24]{hartshorne1977} and
\cite[Tag~0FBG]{stacks}. See the construction in Figure~\ref{fig:normal-cone}.
We blow up $C\times\{0\}$ in $\PP^2\times\mathbb A^1$.
Write $\mathcal Y$ for the resulting threefold, $s$ for the parameter on $\mathbb A^1$,
$\rho:\mathcal Y\to\PP^2$ for the projection, and $\mathcal E$
for the exceptional divisor. The central fiber of $\mathcal Y$ is
$\mathcal Q\cup\mathcal E$, where
\[
 \mathcal Q\simeq\PP^2,\qquad
 \mathcal E\simeq\bar\Sigma,\qquad
 \mathcal Q\cap\mathcal E=C_\infty.
\]
\begin{figure}[htbp]
\centering
\begin{tikzpicture}[x=1cm,y=1cm,font=\small,>=Stealth]
  \node at (3.05,5.95) {\textbf{General fiber: $s\ne0$}};
  \node at (11.35,5.95) {\textbf{Central fiber: $s=0$}};

  \filldraw[fill=diagGray!7,draw=diagGray!65]
    (.2,1.65)--(4.65,1.65)--(5.9,4.05)--(1.45,4.05)--cycle;
  \node at (3.05,1.20) {$\PP^2$};
  \draw[diagBlue,line width=1.1pt]
    (1.3,2.55) arc[start angle=180,end angle=0,
      x radius=1.75,y radius=.62];
  \node[text=diagBlue,inner sep=1pt] at (4.4,2.65) {$C$};
  \coordinate (basepoint) at ({3.05+1.75*cos(120)},{2.55+.62*sin(120)});
  \coordinate (p) at ($(basepoint)+(0,.56)$);
  \fill[diagBlue] (basepoint) circle (2pt);
  \fill[diagOrange] (p) circle (2.2pt);
  \draw[->,diagOrange,line width=.8pt,shorten <=4pt,shorten >=4pt]
    (basepoint)--(p);
  \node[anchor=north,text=diagBlue] at ($(basepoint)+(0,-.15)$) {$x_0$};
  \node[anchor=west,text=diagOrange] at ($(p)+(.15,0)$) {$x(s)$};

  \draw[->,line width=1pt,diagGray] (6.0,3.0)--(8.35,3.0);
  \node[align=center,text width=2.4cm] at (7.18,3.75)
    {Blow up\\$C\times\{0\}$};

  \filldraw[fill=diagGray!8,draw=diagGray!65]
    (8.5,1.65)--(12.95,1.65)--(14.2,4.05)--(9.75,4.05)--cycle;
  \filldraw[fill=diagBlue!5,draw=diagBlue!50]
    (9.6,2.55) arc[start angle=180,end angle=0,
      x radius=1.75,y radius=.62]
    --(13.1,4.7) arc[start angle=0,end angle=180,
      x radius=1.75,y radius=.62]--cycle;
  \foreach \angle in {150,120,90,60,30}
    {\draw[diagBlue!35]
      ({11.35+1.75*cos(\angle)},{2.55+.62*sin(\angle)})
      --++(0,2.15);}
  \draw[diagBlue,line width=1.4pt]
    (9.6,2.55) arc[start angle=180,end angle=0,
      x radius=1.75,y radius=.62];
  \node[fill=diagBlue!5,inner sep=3pt] at (11.35,4.92)
    {$\mathcal E=\bar\Sigma$};
  \node at (11.35,1.20) {$\mathcal Q\simeq\PP^2$};
  \node[text=diagBlue,fill=diagGray!8,inner sep=1pt] at (11.35,2.35)
    {$\mathcal Q\cap\mathcal E=C_\infty$};
  \coordinate (x) at ($(basepoint)+(8.3,.98)$);
  \fill[diagOrange] (x) circle (2.2pt);
  \node[anchor=west,text=diagOrange,fill=diagBlue!5,inner sep=2pt]
    at ($(x)+(.16,0)$) {$x=(x_0,c)$};
\end{tikzpicture}
\caption{Deformation to the normal cone of the cubic.}
\label{fig:normal-cone}
\end{figure}
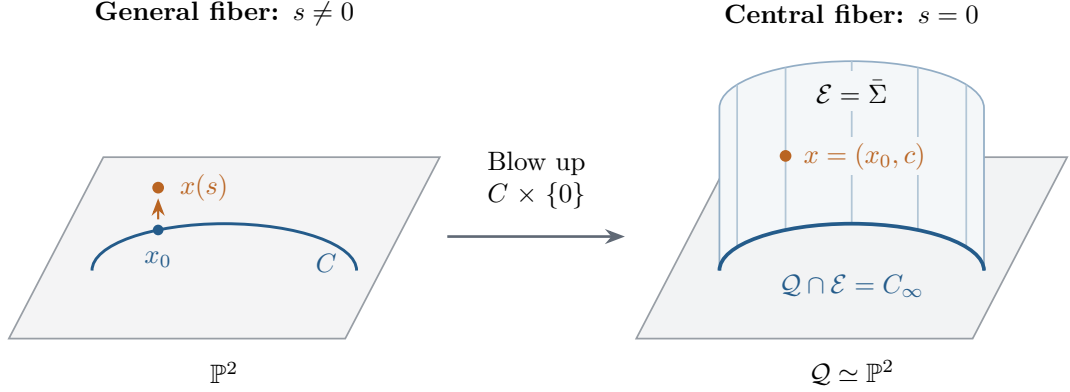

The weighted ideal has finite colength in the holomorphic local ring of $\Tot(M)$.
Choose $N$ such that it
contains $\mathfrak m^{N+1}$, where $\mathfrak m$ is the maximal
ideal of its support. Replace $(w,y)$ by algebraic regular parameters
$(u,v)$ with the same $N$-jet. Their monomials agree modulo
$\mathfrak m^{N+1}$, so they define the same weighted ideal in the holomorphic local ring of $\Tot(M)$.

The morphism $\mathcal Y\setminus\mathcal Q\to\mathbb A^1$ is flat and has smooth fibers.
Therefore, it is smooth.
After an \'etale change of the parameter curve, extend the point of $\Tot(M)$ to
a section of $\mathcal Y\setminus\mathcal Q\to\mathbb A^1$ \cite[Tag~057G]{stacks}.
Lift $u,v$ to regular functions near the section and subtract their
restrictions to it to make them vanish along the section. Their relative differentials are linearly independent
at the original point and hence remain linearly independent along the section after shrinking. They are
therefore relative local parameters along the section. The ideal
$(u^b v^a:pb+qa\ge K)$ defines a finite flat scheme
$\mathcal Z$ supported on this section: its quotient has the fixed
basis $u^b v^a$, $(a,b)\in I_K^{p,q}$.
Its ideal sheaf is flat as well, by the exact sequence
$0\to\mathcal I_{\mathcal Z}\to\OO_{\mathcal Y}\to
\OO_{\mathcal Z}\to0$.

Put $\mathcal A=\rho^*\OO(3h)(-h\mathcal E)$.
Since $\mathcal E|_{\mathcal Q}=C$ and
$\mathcal E+\mathcal Q$ is the divisor of $s$, we obtain
\begin{equation}\label{eq:one-point-twist}
 \mathcal A|_{\mathcal Q}=\OO_{\PP^2},\qquad
 \mathcal A|_{\mathcal E}
 =\OO_{\bar\Sigma}(hC_\infty)\otimes\pi^*M^h
 =\OO_{\bar\Sigma}(hC_0).
\end{equation}
Both restrictions 
$ \mathcal A|_{\mathcal Q}$ and $\mathcal A|_{\mathcal E}$
are trivial on $C_\infty$. A global section of $\mathcal A$ on
$\mathcal Q\cup\mathcal E$ is a pair of sections agreeing along $C_\infty$.
Its restriction to $\mathcal Q$ is a constant, uniquely determined
by the restriction of its $\mathcal E$-part to $C_\infty$.
Since $\mathcal Z_0$ is disjoint from $\mathcal Q$, it follows that
\[
 H^0\!\left(\mathcal Y_0,
    \mathcal A|_{\mathcal Y_0}\otimes\mathcal I_{\mathcal Z_0}\right)
 \simeq
 H^0\!\left(\bar\Sigma,\OO(hC_0)\otimes
                         \mathcal I_{\mathcal Z_0}\right)=0.
\]
Upper semicontinuity \cite[III, Theorem~12.8]{hartshorne1977} gives
the same vanishing on nearby fibers. There the surface is $\PP^2$,
and the vertical twist $\cO(-h\mathcal E)$ restricts trivially. This proves
\eqref{eq:one-point-plane-vanishing}.
\end{proof}

The use of a vertical twist is related to the degeneration method
of Ciliberto--Miranda \cite[\S2]{cm1998}, who blow up a line in the
central plane. Here the cubic gives an elliptic ruled surface.

\begin{proof}[Proof of Theorem~\ref{thm:plane}]
First, let $d=3h$ and suppose $9h^2\le rm^2$, with $r\ge10$.
Apply Theorem~\ref{thm:ruled} and Proposition~\ref{prop:one-point-transfer}
with
\[
 p=1,\qquad q=r,\qquad K=rm\ge3h\sqrt r.
\]
They give local parameters on $\PP^2$ for which
\[
 H^0\!\left(\PP^2,\OO(3h)\otimes(u^r,v)^m\right)=0,
\]
using \eqref{eq:one-point-nagata-ideal}.
Lemma~\ref{lem:collision-open} deforms this scheme into $r$ distinct
multiplicity-$m$ fat points. Injectivity of the evaluation matrix persists, so no degree-$3h$
form vanishes to order at least $m$ at $r$ general points.

For arbitrary positive integers $d,m$ with $d\le m\sqrt r$,
apply this result with $h=d$ and multiplicity $3m$.
If a form of degree $d$ vanished to order at least $m$ at every point,
its cube would be a form of degree $3d$ of multiplicity
at least $3m$, which was just ruled out.
\end{proof}

\begin{example}[Ten and sixteen points]\label{ex:one-point-nagata}
Take $r=10$, $m=3$, and $d=9$, so $h=3$, $K=30$, and $\alpha=1$.
The columns of $S_3(1)$ are
\[
 \begin{array}{c|cccc}
 j&0&1&2&3\\\hline
 k&3,\ldots,29&2,\ldots,19&1,\ldots,9&0.
 \end{array}
\]
They contain $55$ points, matching the dimension of the space of forms of degree $9$.
The staircase $I_{30}^{1,10}$ has column lengths $30,20,10$, so its length
is $60$, the number of conditions imposed by ten triple points.
It follows that no curve of degree $9$ passes through ten general triple points.

Example~\ref{ex:one-point-square} and
Figure~\ref{fig:one-point-triangle} give the case
$r=16$, $m=3$, $d=12$: here $K=rm=48$ and $h=4$.
The argument excludes curves of degree $12$ through sixteen
general triple points. 
\end{example}

\section{Monomial valuations and weighted blow-ups}
\label{sec:one-point}

Throughout this section, $p$ and $q$ are coprime positive integers with $q>9p$.
Fix a point $x\in\PP^2$ and regular local parameters $(u,v)$ at $x$.
The corresponding \emph{monomial valuation} is
\begin{equation}\label{eq:one-point-valuation}
 \nu\!\left(\sum_{a,b\ge0}c_{ab}u^b v^a\right)
   =\min_{c_{ab}\ne0}(pb+qa).
\end{equation}

For a section of $\OO_{\PP^2}(d)$, use any local frame at $x$.
Changing the frame multiplies the local equation by a unit and does not
change its valuation. Put
\begin{equation}\label{adgnsfgnsfgnsfg}
 \widehat\mu(\nu)=
 \sup_{0\ne F\in H^0(\PP^2,\OO(d))\atop d\ge1}\frac{\nu(F)}d.
\end{equation}

\begin{maintheorem}\label{thm:one-point}
For a very general choice of local parameters at the fixed point $x$,
\begin{equation}\label{eq:one-point-strict}
 \nu(F)<d\sqrt{pq}
 \qquad\text{for every }d\ge1\text{ and }0\ne F\in H^0(\PP^2,\OO(d)).
\end{equation}
Nevertheless, 
\begin{equation}\label{adgsdtgsthethw}
\widehat\mu(\nu)=\sqrt{pq}.
\end{equation}
Thus the supremum in \eqref{adgnsfgnsfgnsfg} is not
attained, including when $pq$ is a square.
\end{maintheorem}

If one normalizes the valuation to $\nu/p$ and sets $t=q/p$, the equality \eqref{adgsdtgsthethw}
becomes
$\widehat\mu(\nu/p)=\sqrt t$. This proves the rational-weight case with $t>9$ of the valuative Nagata
conjecture of Dumnicki--Harbourne--K\"uronya--Ro\'e--Szemberg
\cite[Conjecture~4.3]{dhkrs2017}. Their stronger conjecture has the threshold
$t\ge8+1/36$ \cite[Conjecture~2.4]{dhkrs2017}.
For recent equivalent formulations in the broader setting of plane
valuations, see \cite[Conjecture~C and Corollary~3.14]{galindo2025}.
The relation with Nagata's conjecture through collisions of fat
points already appears in \cite[Proposition~4.4]{dhkrs2017}.

It is important to emphasize that the local parameters need to be very general.
For example, in affine linear coordinates the line $v=0$ has valuation
$q>\sqrt{pq}$. To specify a finite-dimensional parameter space, let $\mathfrak m_x$ be
the maximal ideal at $x$ and put $J=\lfloor q/p\rfloor$.
The pairs $(u,v)\in(\mathfrak m_x/\mathfrak m_x^{J+1})^2$ with independent
linear terms form an irreducible open subset, denoted $\mathcal J_J(x)$.
This finite jet determines the valuation in every degree. Indeed,
changing the parameters modulo $\mathfrak m_x^{J+1}$ adds terms of
valuation at least $p(J+1)>q$. The change of parameters therefore
induces the identity on the associated graded ring of the weighted
filtration. In particular, every valuation ideal is unchanged.
We use ``very general'' in this fixed finite-dimensional parameter space;
compare \cite[Remark~2.3]{dhkrs2017}.

\begin{proof}[Proof of Theorem~\ref{thm:one-point}]
For each $h\ge1$, put $K_h=\lceil3h\sqrt{pq}\rceil$.
Theorem~\ref{thm:ruled} and
Proposition~\ref{prop:one-point-transfer} give local parameters on
$\PP^2$ for which the restriction of $H^0(\OO(3h))$ to
$\OO/\mathfrak J_{K_h}^{p,q}$ is injective. A projective automorphism moves
the support of this finite scheme to the fixed point $x$.

For fixed $d,K$, failure of injectivity is a closed condition on
$\mathcal J_J(x)$. Indeed, choose affine coordinates at $x$ and
polynomial representatives of the parameter jets. The coefficients
of the inverse formal change of coordinates, to any fixed order, are
regular functions of these coefficients and the inverse Jacobian
determinant. Expanding the degree-$d$ forms and retaining the terms
indexed by $I_K^{p,q}$ gives a matrix of regular functions. Its maximal
minors define the noninjectivity locus. The valuation ideals are
independent of the chosen representatives.

For $(d,K)=(3h,K_h)$, the local parameters just constructed show that
this closed locus is proper. Choose a jet outside the union of these loci over all $h\ge1$. If a nonzero degree-$d$ form had $\nu(F)\ge d\sqrt{pq}$, then
\[
 \nu(F^3)=3\nu(F)\ge3d\sqrt{pq},\qquad \deg(F^3)=3d.
\]
Since the valuation is integral, $\nu(F^3)\ge K_d$, contradicting
injectivity in degree $3d$. This proves \eqref{eq:one-point-strict}.

For the reverse asymptotic inequality, the length of the weighted
scheme is
\begin{equation}\label{eq:one-point-length}
 \length(\OO/\mathfrak J_K^{p,q})=\sum_{a\ge0}A_a
             =\frac{K^2}{2pq}+O(K),
\end{equation}
where 
$A_a=\max\left\{0,\left\lceil\frac{K-qa}{p}\right\rceil\right\}$ by \eqref{eq:one-point-column-length}.
The estimate follows by summing $(K-qa)/p$ for $0\le a<K/q$; replacing
each term by its ceiling changes the sum by $O(K)$.
For $0<c<\sqrt{pq}$ and $K=\lceil cd\rceil$, the length \eqref{eq:one-point-length} is less than
$h^0(\PP^2,\OO(d))=\binom{d+2}{2}$ for large $d$. Hence some nonzero
degree-$d$ form has valuation at least $K$ for every choice of parameters.
Thus $\widehat\mu(\nu)\ge c$ for every $c<\sqrt{pq}$, proving the
equality claimed in the theorem.
\end{proof}

\begin{example}\label{ex:one-point-contact}
For the weights $(p,q)=(1,10)$ and any $m\ge1$,
\[
 \mathfrak J_{10m}^{1,10}=(u^{10},v)^m,\qquad
 \length(\OO/\mathfrak J_{10m}^{1,10})=5m(m+1).
\]
Writing $F=\sum_{i\ge0}f_i(u)v^i$, we can express these conditions as
$\ord_u f_i\ge10(m-i)$ for $0\le i<m$.

Theorem~\ref{thm:one-point} gives $d>m\sqrt{10}$ for every nonzero
degree-$d$ form satisfying these conditions for very general local parameters.
For $m=3$, this is the ideal in
Example~\ref{ex:one-point-nonsquare}: degree $9$ is excluded,
whereas in degree $10$, the dimension of the space of global forms is $66$, while  there are only $60$ linear conditions.
For $m=5$, degree $15$ is excluded, whereas in degree $16$ the dimension of the space of global forms is
$153$, while there are only $150$ linear conditions.
Thus the least possible degrees in these two cases are $10$ and $16$,
respectively. 
\end{example}

\begin{example}\label{ex:one-point-rational-square}
Take $(p,q)=(4,49)$, so that $q/p=(7/2)^2>9$ and $\sqrt{pq}=14$.
The ideal
\[
 \mathfrak J_{196}^{4,49}
 =(u^{49},u^{37}v,u^{25}v^2,u^{13}v^3,v^4)
\]
has colength $49+37+25+13=124$.
For very general parameters, Theorem~\ref{thm:one-point} excludes
a nonzero form of degree at most $14$ in this ideal.
Degree-fifteen forms, however, have dimension $136>124$.
Hence the least degree satisfying these vanishing conditions is
exactly fifteen.

In particular, no curve of degree $14$ has a local equation
$u^{49}-v^4$ plus terms of higher $(4,49)$-weighted degree
in these prescribed very general parameters.
The model cusp $u^{49}=v^4$ has normalization
$u=t^4$, $v=t^{49}$.
Its delta invariant is the number of gaps in the semigroup
$\langle4,49\rangle$, namely $12+24+36=72$.
This is smaller than the arithmetic genus
$(14-1)(14-2)/2=78$ of a degree-fourteen plane curve.
Thus the genus bound alone does not explain the exclusion.
\end{example}

\paragraph{The weighted blow-up.}
Let $\pi:Y\to\PP^2$ be the normalization of the blow-up of the ideal sheaf
$(u^q,v^p)$ (extended by the unit ideal away from $x$). This is the
$(p,q)$-weighted blow-up of $\PP^2$ with respect to the local parameters $(u,v)$. It is a projective surface with cyclic
quotient singularities and one prime exceptional divisor $E$.
Write $H$ for the pullback of the line class. We have
\begin{equation}\label{eq:one-point-intersections}
 H^2=1,\qquad H\cdot E=0,\qquad E^2=-\frac1{pq}.
\end{equation}
See, for example, 
\cite[Theorem~4.3]{artal2014}.
Every nonexceptional prime divisor of $Y$ is the strict transform of a plane
curve $F=0$ of degree $d$, with class $dH-\nu(F)E$. Put
\[
 D=H-\sqrt{pq}\,E.
\]

\begin{corollary}\label{cor:one-point-cones}
For the very general parameters in Theorem~\ref{thm:one-point},
\begin{equation}\label{eq:one-point-cones}
 \begin{split}
 \operatorname{Nef}(Y)&=\RR_{\ge0}H+\RR_{\ge0}D,\\
 \overline{\operatorname{Eff}}(Y)&=\RR_{\ge0}E+\RR_{\ge0}D.
 \end{split}
\end{equation}
The class $D$ is strictly nef and has square zero. No nonzero
effective $\RR$-divisor has a class on its ray. The Cox ring of $Y$
is not finitely generated. 
\end{corollary}

\begin{proof}
Theorem~\ref{thm:one-point} and \eqref{eq:one-point-intersections} give
\[
 D^2=0,\qquad D\cdot E=\frac1{\sqrt{pq}}>0,\qquad
 D\cdot\widetilde C=d-\frac{\nu(F)}{\sqrt{pq}}>0.
\]
Thus $D$ has positive intersection with every prime curve.
In particular, no nonzero effective $\RR$-divisor can have a class on
its ray. A class $xH-yE$ is nef precisely when $y\ge0$ and
$x\ge y/\sqrt{pq}$. This proves the first formula.
Every effective divisor is a sum of $E$ and strict transforms with
$\nu(F)/d<\sqrt{pq}$. The dimension count in the proof of
Theorem~\ref{thm:one-point} gives effective classes whose ratios
approach $\sqrt{pq}$, proving the second formula.

The divisor class group is generated by $H$ and $E$.
If the Cox ring were finitely generated, the
classes of finitely many homogeneous generators would generate the
closed effective cone. Each extremal ray would then contain a
nonzero effective class, contradicting the assertion about $D$.
\end{proof}

\end{document}